\documentclass[11pt]{article}

\usepackage{enumerate,xspace}
\usepackage{amsmath,amssymb,wasysym}
\usepackage[all]{xy}
\usepackage{proof}
\usepackage[svgnames]{xcolor}
\usepackage{tikz}

\usepackage{stmaryrd} 
\usepackage{mathtools}
\usepackage{latexsym}
\usepackage{dsfont}
\usepackage{multicol}
\usepackage{cmll}
\usepackage{multirow}
\usepackage{longtable}

\usepackage{dirtytalk}

\usepackage{lscape}
\usepackage{array}

\usepackage{hyperref} 
\hypersetup{
    colorlinks,
    citecolor=red,
    filecolor=red,
    linkcolor=blue,
    urlcolor=red
}

\newtheorem{innercustomgeneric}{\customgenericname}
\providecommand{\customgenericname}{}
\newcommand{\newcustomtheorem}[2]{%
  \newenvironment{#1}[1]
  {%
   \renewcommand\customgenericname{#2}%
   \renewcommand\theinnercustomgeneric{##1}%
   \innercustomgeneric
  }
  {\endinnercustomgeneric}
}

\newcustomtheorem{customthm}{Theorem}

\newtheorem{observation}{Remark}[section]
\newtheorem{lemma}[observation]{Lemma}  
\newtheorem{theorem}[observation]{Theorem}
\newtheorem{definition}[observation]{Definition}

\newtheorem{proposition}[observation]{Proposition} 
\newtheorem{corollary}[observation]{Corollary}

\makeatletter

\newdimen\w@dth

\def\setw@dth#1#2{\setbox\z@\hbox{\scriptsize $#1$}\w@dth=\wd\z@
\setbox\@ne\hbox{\scriptsize $#2$}\ifnum\w@dth<\wd\@ne \w@dth=\wd\@ne \fi
\advance\w@dth by 1.2em}

\def\t@^#1_#2{\allowbreak\def\n@one{#1}\def\n@two{#2}\mathrel
{\setw@dth{#1}{#2}
\mathop{\hbox to \w@dth{\rightarrowfill}}\limits
\ifx\n@one\empty\else ^{\box\z@}\fi
\ifx\n@two\empty\else _{\box\@ne}\fi}}
\def\t@@^#1{\@ifnextchar_ {\t@^{#1}}{\t@^{#1}_{}}}

\def\t@left^#1_#2{\def\n@one{#1}\def\n@two{#2}\mathrel{\setw@dth{#1}{#2}
\mathop{\hbox to \w@dth{\leftarrowfill}}\limits
\ifx\n@one\empty\else ^{\box\z@}\fi
\ifx\n@two\empty\else _{\box\@ne}\fi}}
\def\t@@left^#1{\@ifnextchar_ {\t@left^{#1}}{\t@left^{#1}_{}}}

\def\two@^#1_#2{\def\n@one{#1}\def\n@two{#2}\mathrel{\setw@dth{#1}{#2}
\mathop{\vcenter{\hbox to \w@dth{\rightarrowfill}\kern-1.7ex
                 \hbox to \w@dth{\rightarrowfill}}%
       }\limits
\ifx\n@one\empty\else ^{\box\z@}\fi
\ifx\n@two\empty\else _{\box\@ne}\fi}}
\def\tw@@^#1{\@ifnextchar_ {\two@^{#1}}{\two@^{#1}_{}}}

\def\tofr@^#1_#2{\def\n@one{#1}\def\n@two{#2}\mathrel{\setw@dth{#1}{#2}
\mathop{\vcenter{\hbox to \w@dth{\rightarrowfill}\kern-1.7ex
                 \hbox to \w@dth{\leftarrowfill}}%
       }\limits
\ifx\n@one\empty\else ^{\box\z@}\fi
\ifx\n@two\empty\else _{\box\@ne}\fi}}
\def\t@fr@^#1{\@ifnextchar_ {\tofr@^{#1}}{\tofr@^{#1}_{}}}

\newdimen\W@dth
\def\setW@dth#1#2{\setbox\z@\hbox{$#1$}\W@dth=\wd\z@
\setbox\@ne\hbox{$#2$}\ifnum\W@dth<\wd\@ne \W@dth=\wd\@ne \fi
\advance\W@dth by 1.2em}

\def\T@^#1_#2{\allowbreak\def\N@one{#1}\def\N@two{#2}\mathrel
{\setW@dth{#1}{#2}
\mathop{\hbox to \W@dth{\rightarrowfill}}\limits
\ifx\N@one\empty\else ^{\box\z@}\fi
\ifx\N@two\empty\else _{\box\@ne}\fi}}
\def\T@@^#1{\@ifnextchar_ {\T@^{#1}}{\T@^{#1}_{}}}

\def\T@left^#1_#2{\def\N@one{#1}\def\N@two{#2}\mathrel{\setW@dth{#1}{#2}
\mathop{\hbox to \W@dth{\leftarrowfill}}\limits
\ifx\N@one\empty\else ^{\box\z@}\fi
\ifx\N@two\empty\else _{\box\@ne}\fi}}
\def\T@@left^#1{\@ifnextchar_ {\T@left^{#1}}{\T@left^{#1}_{}}}

\def\Tofr@^#1_#2{\def\N@one{#1}\def\N@two{#2}\mathrel{\setW@dth{#1}{#2}
\mathop{\vcenter{\hbox to \W@dth{\rightarrowfill}\kern-1.7ex
                 \hbox to \W@dth{\leftarrowfill}}%
       }\limits
\ifx\N@one\empty\else ^{\box\z@}\fi
\ifx\N@two\empty\else _{\box\@ne}\fi}}
\def\T@fr@^#1{\@ifnextchar_ {\Tofr@^{#1}}{\Tofr@^{#1}_{}}}

\def\Two@^#1_#2{\def\N@one{#1}\def\N@two{#2}\mathrel{\setW@dth{#1}{#2}
\mathop{\vcenter{\hbox to \W@dth{\rightarrowfill}\kern-1.7ex
                 \hbox to \W@dth{\rightarrowfill}}%
       }\limits
\ifx\N@one\empty\else ^{\box\z@}\fi
\ifx\N@two\empty\else _{\box\@ne}\fi}}
\def\Tw@@^#1{\@ifnextchar_ {\Two@^{#1}}{\Two@^{#1}_{}}}

\def\to{\@ifnextchar^ {\t@@}{\t@@^{}}}
\def\from{\@ifnextchar^ {\t@@left}{\t@@left^{}}}
\def\tofro{\@ifnextchar^ {\t@fr@}{\t@fr@^{}}}
\def\To{\@ifnextchar^ {\T@@}{\T@@^{}}}
\def\From{\@ifnextchar^ {\T@@left}{\T@@left^{}}}
\def\Two{\@ifnextchar^ {\Tw@@}{\Tw@@^{}}}
\def\Tofro{\@ifnextchar^ {\T@fr@}{\T@fr@^{}}}

\makeatother

\title{Convenient Antiderivatives For Differential Linear Categories}
\author{Jean-Simon Pacaud Lemay}

\begin{document}
\allowdisplaybreaks

\maketitle



\begin{abstract} Differential categories axiomatize the basics of differentiation and provide categorical models of differential linear logic. A differential category is said to have antiderivatives if a natural transformation $\mathsf{K}$, which all differential categories have, is a natural isomorphism. Differential categories with antiderivatives come equipped with a canonical integration operator such that generalizations of the Fundamental Theorems of Calculus hold. In this paper, we show that Blute, Ehrhard, and Tasson's differential category of convenient vector spaces has antiderivatives. To help prove this result, we show that a differential linear category -- which is a differential category with a monoidal coalgebra modality -- has antiderivatives if and only if one can integrate over the monoidal unit and such that the Fundamental Theorems of Calculus hold. We also show that generalizations of the relational model (which are biproduct completions of complete semirings) are also differential linear categories with antiderivatives.
\end{abstract}

\subparagraph*{Acknowledgements.} The author would like to thank Robin Cockett, Geoff Cruttwell, Thomas Ehrhard, Christine Tasson, and the anonymous reviewers for useful discussions and editorial comments. The author also thanks Kellogg College, the Clarendon Fund, and the Oxford-Google DeepMind Graduate Scholarship for financial support. 

\section{Introduction}

In single-variable calculus, the relationship between differentiation and integration is captured by the two Fundamental Theorems of Calculus, which in particular relates antiderivatives to definite integrals. The First Fundamental Theorem of Calculus states that the bounded integral of a smooth function is one of its antiderivatives, that is, the derivative of the integral of a function is equal to the original function:
\begin{align*}
\frac{{\sf d} (\int_a^t f(u)~{\sf d}u)}{{\sf d}t}(x)=f(x)
\end{align*}
While the Second Fundamental Theorem of Calculus directly relates the derivative and the Riemann integral in the following way: for any differential function $f: \mathbb{R} \to \mathbb{R}$, the Riemann integral of its derivative over an interval $[a,b]$ is given by the difference at the endpoints of $f$: 
\begin{align*}
\int_a^b\frac{{\sf d} f(t)}{{\sf d}t}(s)~{\sf d}s=f(b)-f(a)
\end{align*}
The generalization of the Second Fundamental Theorem of Calculus to the multivariable setting is given by the Fundamental Theorem of Line Integrals (also sometimes known as the Gradient Theorem) which, as the names suggest, relates line integration to the gradient. 
Given a vector field $F: \mathbb{R}^n \to \mathbb{R}^n$, recall that its line integral over a curve $\mathcal{C}$ parametrized by $\mathsf{r}: [0,1] \to \mathbb{R}^n$ is defined as follows: 
\[\int \limits_C F\cdot \mathsf{d}\mathsf{r} := \int \limits^b_a F(\mathsf{r}(t)) \cdot \mathsf{r}^\prime (t) ~ \mathsf{d}t \]
Note that while line integration (and its iterated versions) is a concept for multivariable calculus, the line integral itself is computed as an integral over the real as in the single-variable setting. In a certain sense, even the notion of integrating differential forms over manifolds (a fundamental concept of integration in differential geometry) comes down simply to integrating over the real line (especially for Riemannian manifolds). This still true in categorical models of differential linear logic, in that the concepts of integration and antiderivatives in this setting are solely dependent on the ability to integrate over the monoidal unit. 

Differential categories were introduced by Blute, Cockett, and Seely in \cite{blute2006differential} to provide categorical models of differential linear logic \cite{ehrhard2017introduction}. As such, differential categories provide an algebraic axiomatization of the basic foundations of differentiation. The coKleisli category of a differential category is a Cartesian differential category \cite{blute2009cartesian}, which axiomatizes the directional derivative and differential calculus on Euclidean spaces, and also provides categorical models of the differential $\lambda$-calculus, as introduced by Ehrhard and Regnier in \cite{ehrhard2003differential}. Differential categories now have a rich literature with many interesting examples such as commutative algebras, $\mathcal{C}^\infty$-rings, finiteness spaces, Rota-Baxter algebras, K\"othe spaces, etc. One particular example with close ties to differential geometry is the differential category of convenient vector spaces, introduced by Blute, Ehrhard, and Tasson in \cite{blute2010convenient}. 

Convenient vector spaces, introduced by Fr\"{o}licher and Kriegl in \cite{frolicher1988linear}, have been used to study differential geometry on infinite-dimensional manifolds since convenient vector spaces have many desirable and well-behaved properties \cite{kriegl1997convenient}. In particular, the category of convenient vector spaces and smooth maps between them, $\mathsf{CON}_{sm}$, is a Cartesian closed category \cite[Theorem 3.12]{kriegl1997convenient}, unlike other categories related to differential geometry such as the category of smooth manifolds. Furthermore, $\mathsf{CON}_{sm}$ is isomorphic to the coKleisli of a comonad on the category of convenient vector spaces and bounded linear maps, $\mathsf{CON}$ \cite[Theorem 6.3]{blute2010convenient}. In fact, this comonad is a coalgebra modality which has the Seely isomorphisms \cite[Lemma 6.4]{blute2010convenient} and a deriving transformation \cite[Theorem 6.6]{blute2010convenient}. Therefore $\mathsf{CON}$ is a differential category, and as a consequence $\mathsf{CON}_{sm}$ is a Cartesian differential category. In their conclusion, Blute, Ehrhard, and Tasson state: \say{$\hdots$ a next fundamental question is the logical/syntactic structure of integration. One would like an integral linear logic, which would again treat integration as an inference rule. It should not be a surprise at this point that convenient vector spaces are extremely well-behaved with regards to integration. The category [of convenient vector spaces] will likely provide an excellent indicator of the appropriate structure.} While such a categorical framework for integration has been developed, one has not yet gone back to check that the differential category of convenient vector spaces provides a model of this integration. 

The notion of integration in a differential category was first introduced by Ehrhard \cite{ehrhard2017introduction}, while an axiomatization of integration separate from differentiation was later developed by Cockett and Lemay with the introduction of integral categories \cite{cockett_lemay_2018}. Somewhat analogue to differential categories, the axioms of an integral category are the basic rules of integration which include that the integral of a constant function is a linear function and the Rota-Baxter rule \cite{guo2012introduction}, which is an expression of integration by parts using only integrals. The coKleisli categories of appropriate integral categories are known as Cartesian integral categories \cite{COCKETT201845}, which takes a more analytic approach. Axiomatizing integration in this manner has also lead to studying the Fundamental Theorems of Calculus in the differential category setting. A calculus category \cite[Definition 5.6]{cockett_lemay_2018} is a differential category which is also an integral category such that the differential structure and integral structure are compatible in the sense of satisfying analogue versions of both Fundamental Theorems of Calculus. In particular, as previously mentioned, the Fundamental Theorems of Calculus link integrals to antiderivatives and vice-versa. This leads to the concept of a differential category having antiderivatives \cite{cockett_lemay_2018,ehrhard2017introduction}, which is a way of obtaining an integral structure from the differential structure. Explicitly, a differential category is said to have {antiderivatives \footnote{In the Cockett and Lemay sense, which implies Ehrhard's notion of having antiderivatives}} if a natural transformation $\mathsf{K}$, which all differential categories have, is a natural isomorphism. Furthermore, every differential category with antiderivatives is a calculus category with respect to the integral constructed using the inverse of $\mathsf{K}$ (Definition \ref{ASdef}). The main objective of this paper is to show that the differential category of convenient vector spaces admits antiderivatives and therefore admits an integral structure such that the Fundamental Theorems of Calculus hold. 

To help us show that convenient vector spaces provide a differential category with antiderivatives, we will need to take a closer look at when differential linear categories have antiderivatives. Indeed, if one were to charge headfirst into proving that $\mathsf{K}$ was an isomorphism, one would have to deal with infinite-dimensional convenient vector spaces and many technical analytic nuances. However for differential categories with a monoidal coalgebra modality, which we call here a differential linear category, one can give a simple sufficient condition for having antiderivatives. As discussed at the beginning of the introduction, it turns out that it is sufficient to be able to integrate over the monoidal unit and also check that the Second Fundamental Theorem of Calculus holds (Theorem \ref{antithm}). This greatly simplifies showing that a differential linear category has antiderivatives, as one only needs to work with the monoidal unit. The idea here is the same as for line integration in the sense that while one can define integration for any object, one is only really integrating over the monoidal unit. In the case of convenient vector spaces, the monoidal unit is $\mathbb{R}$, which is especially well behaved and easy to work with. This general result for differential linear categories implies that to be able to integrate and obtain antiderivatives for arbitrary smooth maps between convenient vector spaces, one only needs to understand how to integrate smooth curves $\mathsf{c}: \mathbb{R} \to \mathbb{R}$, which is a key concept in the theory of convenient vector spaces. In fact, for a convenient vector space, every smooth curve admits an antiderivative and also the Second Fundamental Theorem of Calculus holds. This observation and Theorem \ref{antithm} is essentially the proof that the differential category of convenient vector spaces admits antiderivatives (Section \ref{consec}). Furthermore, as another application of Theorem \ref{antithm}, we are also able to show that weighted generalizations of the relational model (which are biproduct completions of complete semirings) are also differential categories with antiderivatives (Section \ref{bisec}).

\subparagraph*{Main Results:} 

The main technical result of this paper is the following: 

\begin{customthm}{\ref{antithm}} A differential linear category has antiderivatives if and only if for the monoidal unit $R$ there is a map $\mathsf{s}_R: \oc R \to \oc R$ such that $\mathsf{s}_R$ and the deriving transformation $\mathsf{d}_R: \oc R \to \oc R$ satisfy the Second Fundamental Theorem of Calculus, that is, $\mathsf{s}_R \mathsf{d}_R + \oc(0) = 1_{\oc R}$. 
\end{customthm} 

The above theorem simplifies proving the main goal of this paper which is the following: 

\begin{customthm}{\ref{mainthm}} $\mathsf{CON}$, the category of convenient vector spaces and bounded linear maps between them, is a differential linear category with antiderivatives. 
\end{customthm} 

\noindent \subparagraph*{Outline:} We begin with Section \ref{diffsec} which provides a recap of differential (linear) categories with antiderivatives. Afterwards, in Section \ref{diffantisec} we briefly review differential categories with antiderivatives. In Section \ref{difflinsec} we study differential linear categories with antiderivatives and in particular prove Theorem \ref{antithm}, the main technical result of this paper. Sections \ref{polysec}, \ref{bisec}, and \ref{consec} are dedicated to providing examples of differential linear categories with antiderivatives by applying Theorem \ref{antithm}. In particular, Section \ref{consec} is dedicated to the main goal of this paper of showing that the differential category of convenient vector spaces has antiderivatives. 

\subparagraph*{Conventions:} In these notes, we will use diagrammatic order for composition. Explicitly, this means that the composite map $fg$ is the map which first does $f$ then $g$. Also, to simplify working in a symmetric monoidal category, we will instead work in a symmetric \emph{strict} monoidal category \cite{mac2013categories}, that is, the unit and associativity isomorphisms are identities. We denote symmetric monoidal categories as quadruples $(\mathbb{X}, \otimes, R, \sigma)$ where $\mathbb{X}$ is the base category, $\otimes$ is the tensor product, $R$ for the monoidal unit (so in particular $A \otimes R = A = R \otimes A$), and ${\sigma_{A,B}: A \otimes B \to B \otimes A}$ is the symmetry natural isomorphism. 

\section{Differential Categories and Differential Linear Categories}\label{diffsec}

In this section, we give a brief overview of differential categories and differential linear categories. We begin by recalling the notion of coalgebra modalities (Definition \ref{coalgdef}) and the coderiving transformation (Definition \ref{dcircdef}). Then we review differential categories (Definition \ref{diffdef}) and in particular we also discuss the natural transformations $\mathsf{K}$ and $\mathsf{J}$ (Definition \ref{KJdef}), both of which play fundamental roles for the notion of antiderivatives (which we discuss in the next section). Afterwards, we consider monoidal coalgebra modalities (Definition \ref{moncoalgmod}), the Seely isomorphisms (Definition \ref{Seelydef}), and differential linear categories (Definition \ref{difflindef}). For a more complete story on differential (linear) categories, including the relevant commutative diagrams and string diagram representations, we refer the reader to \cite{Blute2019,blute2006differential}. 

Coalgebra modalities \cite{blute2006differential} are comonads $\oc$ such that for each object $A$, $\oc A$ comes equipped with a natural cocommutative comonoid structure. Coalgebra modalities are strictly weaker structure then what is required for a categorical model of the multiplicative and exponential fragment of linear logic ($\mathsf{MELL}$) \cite{bierman1995categorical,mellies2009categorical}, for that one requires a \emph{monoidal} coalgebra modality. However, coalgebra modalities are sufficient to axiomatize differentiation. 

\begin{definition}\label{coalgdef} A \textbf{coalgebra modality} \cite[Definition 2.1]{blute2006differential} on a symmetric monoidal category $(\mathbb{X}, \otimes, R, \sigma)$ is a quintuple $(\oc, \delta, \varepsilon, \Delta, e)$ consisting of a functor ${\oc: \mathbb{X} \to \mathbb{X}}$ and four natural transformations ${\rho_A: \oc A \to \oc \oc A}$, $ \varepsilon_A: \oc A \to A$, $\Delta_A: \oc A \to \oc A \otimes \oc A $ and $e_A: \oc A \to R$ such that: 
\begin{enumerate}[{\em (i)}]
\item $(\oc, \delta, \varepsilon)$ is a comonad on $\mathbb{X}$, that is, the following equalities hold:
\begin{align*} \rho_A \varepsilon_{\oc A} = 1_{\oc A} = \rho_A \oc(\varepsilon_A) && \rho_A \rho_{\oc A} = \rho_A \oc(\rho_A) 
\end{align*}
\item $(\oc A, \Delta_A, e_A)$ is a cocommutative comonoid, that is, the following equalities hold: 
\begin{align*} \Delta_A (\Delta_A \otimes 1_{\oc A} = \Delta_A(1_{\oc A} \otimes \Delta_A) && \Delta_A(e_A \otimes 1_{\oc A}) = 1_{\oc A} = \Delta(1_{\oc A} \otimes e_A)
\end{align*}
\[\Delta_A \sigma_{\oc A, \oc A} = \Delta_A\] 
\item $\delta_A$ is a comonoid morphism, that is, the following diagram commutes: 
\begin{align*} \Delta_A (\rho_A \otimes \rho_A) = \rho_A \Delta_{\oc A} && \rho_A e_{\oc A} = e_A 
\end{align*}
\end{enumerate}
 \end{definition}

CoKleisli maps of coalgebra modalities, that is, maps of type $f: \oc A \to B$, are of particular interest as they should be thought of as \emph{smooth} maps. This terminology is of no coincidence. Indeed, in a differential category, the differentiable maps are precisely the coKleisli maps, and they are (in a certain way) infinitely differentiable and hence smooth. A subclass of these smooth maps are the \emph{linear} maps which are coKleisli maps of the form $\varepsilon_A g: \oc A \to B$ for some map $g: A \to B$.  

Every coalgebra modality comes equipped with an important natural transformation known as the coderiving transformation -- which plays a central role in the integration side of the story.

\begin{definition} \label{dcircdef} For a coalgebra modality $(\oc, \rho, \varepsilon, \Delta, \mathsf{e})$, the \textbf{coderiving transformation} \cite[Definition 2.2]{cockett_lemay_2018} is the natural transformation $\mathsf{d}^\circ_A: \oc A \to \oc A \otimes A$ defined as the following composite: 
\begin{equation}\label{coderive}\begin{gathered} \xymatrixcolsep{5pc}\xymatrix{\mathsf{d}^\circ_A := \oc A \ar[r]^-{\Delta_A} & \oc A \otimes \oc A \ar[r]^-{1_{\oc A} \otimes \varepsilon_A} & \oc A \otimes A
 } \end{gathered}\end{equation}
\end{definition}
For a list of identities the coderiving transformation satisfies see \cite[Proposition 2.1]{cockett_lemay_2018}. 

Differential categories were introduced by Blute, Cockett, and Seely in \cite{blute2006differential} to provide an algebraic axiomatization of the basic properties of the differentiation. Two of the basic properties of the derivative from classical differential calculus requires addition: that the derivative of a constant function is zero and the Leibniz rule for deriving a product of functions. Therefore, we must first discuss the basic additive structure of a differential category which is captured by the notion of additive symmetric monoidal categories. Here we mean ``additive'' in the Blute, Cockett, and Seely sense of the term \cite{blute2006differential}, that is, to mean enriched over commutative monoids. In particular, we do not assume negatives nor do we assume biproducts (which differs from other definitions of an additive category found in the literature). 

\begin{definition}\label{addcatdef} An \textbf{additive category} \cite[Definition 3]{Blute2019} is a commutative monoid enriched category, that is, a category $\mathbb{X}$ in which each hom-set $\mathbb{X}(A,B)$ is a commutative monoid with an addition operation $+: \mathbb{X}(A,B) \times \mathbb{X}(A,B) \to \mathbb{X}(A,B)$ and a zero $0 \in \mathbb{X}(A,B)$, and such that composition preserves the additive structure, that is:
\begin{align*}
k(f\!+\!g)h\!=kfh\!+\!kgh && k0h=0
\end{align*}
An \textbf{additive symmetric monoidal category} \cite[Definition 3]{Blute2019} is a symmetric monoidal category $(\mathbb{X}, \otimes, R, \sigma)$ such that $\mathbb{X}$ is also an additive category in which the tensor product $\otimes$ is compatible with the additive structure in the sense that: 
\begin{align*}
k \otimes (f\!+\!g)\otimes h\!= \!k\otimes\!f\otimes h \!+ \!k\otimes\!g\otimes h && k \otimes 0 \otimes h\!=\!0 
\end{align*}
\end{definition}

It is worth mentioning that every additive category can be completed to a category with finite biproducts (which is itself an additive category), and similarly, every additive symmetric monoidal category can be completed to an additive symmetric monoidal category with finite biproducts. For this reason, it can be argued that that one should always assume a setting with finite biproducts, such as in \cite{fiore2007differential}. The problem is that arbitrary coalgebra modalities do not necessarily extend to the finite biproduct completion. On the other hand, monoidal coalgebra modalities induce monoidal coalgebra modalities on the finite biproduct completion \cite[Section 7]{Blute2019}. However, finite biproducts do not play an important technical role in this paper, so we will continue without them. 

\begin{definition}\label{diffdef} A \textbf{differential category} \cite[Definition 2.4]{blute2006differential} is an additive symmetric monoidal category $(\mathbb{X}, \otimes, R, \sigma)$ with a coalgebra modality $(\oc, \rho, \varepsilon, \Delta, \mathsf{e})$ which comes equipped with a \textbf{deriving transformation} \cite[Definition 7]{Blute2019}, that is, a natural transformation $\mathsf{d}_A: \oc A \otimes A \to \oc A$ such that the following equalities hold: 
\begin{enumerate}[{\bf [d.1]}]
\item \textbf{Constant Rule}: $\mathsf{d}_A \mathsf{e}_A = 0$
\item \textbf{Leibniz Rule}: $\mathsf{d}_A \Delta_A=(\Delta_A \otimes 1_{A})(1_{\oc A} \otimes \sigma_{\oc A,A})(\mathsf{d}_A \otimes 1_{\oc A})+ (\Delta_A \otimes 1_A)(1_{\oc A} \otimes \mathsf{d}_A) $
\item \textbf{Linear Rule}: $\mathsf{d}_A \varepsilon_A =\mathsf{e}_A \otimes 1_{\oc A}$
\item \textbf{Chain Rule}: $\mathsf{d}_A\rho_A=(\Delta_A \otimes 1_A)(\rho_A \otimes \mathsf{d}_A)\mathsf{d}_{\oc A}$
\item \textbf{Interchange Rule}: $(\mathsf{d}_A \otimes 1_A)\mathsf{d}_A=(1_{\oc A} \otimes \sigma_{A,A})(\mathsf{d}_A \otimes 1_A)\mathsf{d}_A$
\end{enumerate}
\end{definition}

The derivative of a coKleisli map $f: \oc A \to B$ (which recall are interpreted as smooth maps) is the map $\mathsf{D}[f]: \oc A \otimes A \to B$, defined as the composite $\mathsf{D}[f]:= \mathsf{d}_A f$. The first deriving transformation axiom, the constant rule {\bf [d.1]}, states that the derivative of a constant map is zero. The second axiom {\bf [d.2]} is the Leibniz rule for differentiation. The third axiom, the linear rule {\bf [d.3]}, says that the derivative of a linear map (which recall are maps of the form $\varepsilon_A g$) is a constant. The fourth axiom {\bf [d.4]} is the chain rule, describing how to differentiate composition in the coKleisli category. And the last axiom, the interchange rule {\bf [d.5]}, is the independence of differentiation, which naively states that differentiating with respect to $x$ then $y$ is the same as differentiation with respect to $y$ then $x$. It should be noted that the interchange rule {\bf [d.5]} was not part of the definition in \cite[Definition 2.5]{blute2006differential} but was later added to ensure that the coKleisli category of a differential category was a Cartesian differential category \cite[Proposition 3.2.1]{blute2009cartesian}. Many examples of differential categories can be found throughout the literature, such as in \cite[Section 9]{Blute2019}. 

By the Leibniz Rule {\bf [d.2]} and the Linear Rule {\bf [d.3]}, the deriving transformation and coderiving transformation (Definition \ref{dcircdef}) are compatible in the following sense: 

\begin{proposition}\label{Wp} \cite[Proposition 4.1]{cockett_lemay_2018} In a differential category, the deriving transformation $\mathsf{d}$ and coderiving transformation $\mathsf{d}^\circ$ satisfy the following equality:
\begin{equation}\label{W}\begin{gathered} \mathsf{d}_A\mathsf{d}^\circ_A = (\mathsf{d}^\circ_A \otimes 1_A)(1_A \otimes \sigma_{A,A})(\mathsf{d}_A \otimes 1_A) + (1_{\oc A} \otimes 1_A) \end{gathered}\end{equation}
\end{proposition} 

In every differential category, there are two important natural transformations which are constructed using both the deriving transformation and coderiving transformation:

\begin{definition}\label{KJdef} In a differential category, define the natural transformations $\mathsf{K}_A: \oc A \to \oc A$ \cite[Definition 4.2]{cockett_lemay_2018} and $\mathsf{J}_A: \oc A \to \oc A$ \cite [Section 3.2]{ehrhard2017introduction} respectively as follows: 
\begin{equation}\label{Kdef}\begin{gathered} \mathsf{K}_A := \left( \xymatrixcolsep{5pc}\xymatrix{ \oc A \ar[r]^-{\mathsf{d}^\circ_A} & \oc A \otimes A \ar[r]^-{\mathsf{d}_A} & \oc A} \right) + \left( \xymatrixcolsep{5pc}\xymatrix{ \oc A \ar[r]^-{\oc(0)} & \oc A} \right) \end{gathered}\end{equation}
\begin{equation}\label{Kdef}\begin{gathered} \mathsf{J}_A := \left( \xymatrixcolsep{5pc}\xymatrix{ \oc A \ar[r]^-{\mathsf{d}^\circ_A} & \oc A \otimes A \ar[r]^-{\mathsf{d}_A} & \oc A} \right) + \left( \xymatrixcolsep{5pc}\xymatrix{ \oc A \ar@{=}[r]& \oc A} \right) \end{gathered}\end{equation}
\end{definition}

\noindent For a list of identities which $\mathsf{K}$ and $\mathsf{J}$ satisfy see \cite[Corollary 4.1, Proposition 4.4]{cockett_lemay_2018}. In particular, $\mathsf{K}$ and $\mathsf{J}$ are related by the following identities: 

\begin{proposition}\label{KJp} \cite[Proposition 4.4]{cockett_lemay_2018} In a differential category, the following equalities hold:
\begin{align*}
\mathsf{K}_A \oc(0) = \oc(0) = \oc(0) \mathsf{K}_A && \mathsf{J}_A \oc(0) = \oc(0) = \oc(0) \mathsf{J}_A\\ 
\mathsf{K}_A\mathsf{d}^\circ_A = \mathsf{d}^\circ_A (\mathsf{J}_A \otimes 1_A) && \mathsf{d}_A\mathsf{K}_A = (\mathsf{J}_A \otimes 1_A)\mathsf{d}_A
\end{align*} 
\end{proposition} 

We now turn our attention to differential categories with \emph{monoidal} coalgebra modalities. Let us first recall the notion of a monoidal coalgebra modality -- also sometimes known as a \textbf{linear exponential modality} \cite{schalk2004categorical}. Monoidal coalgebra modalities are coalgebra modalities whose underlying comonad is also a symmetric monoidal comonad. Symmetric monoidal closed categories with a monoidal coalgebra are categorical models of $\mathsf{MELL}$ -- also known as linear categories \cite{bierman1995categorical,mellies2009categorical}.

 \begin{definition}\label{moncoalgmod} A \textbf{monoidal coalgebra modality} \cite[Definition 2]{Blute2019} on a symmetric monoidal category $(\mathbb{X}, \otimes, R, \sigma)$ is a septuple $(\oc, \rho, \varepsilon, \Delta, \mathsf{e}, \mathsf{m}, \mathsf{m}_R)$ consisting of a coalgebra modality $(\oc, \rho, \varepsilon, \Delta, e)$ and a natural transformation ${\mathsf{m}_{A,B}: \oc A \otimes \oc B\to \oc(A \otimes B)}$, and a map ${\mathsf{m}_R: R\to \oc R}$ such that:
 \begin{enumerate}[{\em (i)}]
\item $(\oc, \mathsf{m}, \mathsf{m}_R)$ is a symmetric monoidal functor, that is, the following equalities hold: 
\[(\mathsf{m}_{A,B} \otimes 1_{\oc C})\mathsf{m}_{A\otimes B, C} = (1_{\oc A} \otimes \mathsf{m}_{B,C})\mathsf{m}_{A, B \otimes C} \] 
\begin{align*} 
 (\mathsf{m}_R \otimes 1_{\oc A})\mathsf{m}_{R,A} = 1_{\oc A} = (1_{\oc A} \otimes \mathsf{m}_R)\mathsf{m} _{A,R} && \sigma_{\oc A,\oc B} \mathsf{m}_{B,A} = \mathsf{m}_{A,B} \oc(\sigma_{A,B}) 
\end{align*}
\item $\rho$ and $\varepsilon$ are monoidal transformations, that is, the equalities hold: 
\begin{align*} 
 \mathsf{m}_{A,B}\rho_{A \otimes B} = (\rho_A \otimes \rho_B)\mathsf{m}_{\oc A,\oc B}\oc(\mathsf{m}_{A,B}) && \mathsf{m}_{A,B}\varepsilon_{A \otimes B} = \varepsilon_A \otimes \varepsilon_B 
 \end{align*}
 \begin{align*} 
 \mathsf{m}_R\rho_R = \mathsf{m}_R \oc(\mathsf{m}_R) && \mathsf{m}_R \varepsilon_R = 1_R 
\end{align*}
\item $\Delta$ and $e$ are monoidal transformations, that is, the following equalities hold: 
\begin{align*} 
\mathsf{m}_{A,B} \Delta_{A \otimes B} = (\Delta_A \otimes \Delta_B)(1_{\oc A} \otimes \sigma_{\oc A, \oc B} \otimes 1_{\oc B})(\mathsf{m}_{A,B} \otimes \mathsf{m}_{A,B}) && \mathsf{m}_{A,B}\mathsf{e}_{A \otimes B}= \mathsf{e}_A \otimes \mathsf{e}_B 
\end{align*}
\begin{align*} 
 \mathsf{m}_R \Delta_R = \mathsf{m}_R \otimes \mathsf{m}_R && \mathsf{m}_R \mathsf{e}_R = 1_R 
\end{align*}
\item $\Delta$ and $e$ are $\oc$-coalgebra morphisms, that is, the following equalities hold:
\begin{align*}  
\rho_A \oc(\Delta_A) = \Delta_A (\rho_A \otimes \rho_A) \mathsf{m}_{\oc A, \oc A} && \rho_A \oc(\mathsf{e}_A) = \mathsf{e}_A \mathsf{m}_R
\end{align*}
\end{enumerate}
A \textbf{linear category} \cite[Definition 2]{Blute2019} is a symmetric monoidal category with a monoidal coalgebra modality. 
\end{definition}

We should note that here we are using the term ``linear category'' in the sense of Blute, Cockett, and Seely as in \cite{blute2015cartesian}, which is the same as Bierman's definition in \cite{bierman1995categorical} but which drops the closed structure requirement. Many examples of monoidal coalgebra modalities can be found throughout the literature, since every categorical model of $\mathsf{MELL}$ admits a monoidal coalgebra modality. For example, Hyland and Schalk provide a nice list of examples in \cite[Section 2.4]{hyland2003glueing}. Examples of coalgebra modalities that are not monoidal can be found in \cite[Section 9]{Blute2019}. 

\begin{proposition}\label{dcircm} \cite[Proposition 2.2]{cockett_lemay_2018} For a monoidal coalgebra modality $(\oc, \rho, \varepsilon, \Delta, \mathsf{e}, \mathsf{m}, \mathsf{m}_R)$, its coderiving transformation $\mathsf{d}^\circ$ (Definition \ref{dcircdef}) satisfies the following equalities: 
\begin{align*}  
\mathsf{m}_{A,B}\mathsf{d}^\circ_{A \otimes B} = (\mathsf{d}^\circ_A \otimes \mathsf{d}^\circ_B)(1_{\oc A} \otimes \sigma_{A, \oc B} \otimes 1_B)(\mathsf{m}_{A,B} \otimes 1_A \otimes 1_B) && \mathsf{m}_R \mathsf{d}^\circ_R = \mathsf{m}_R
\end{align*}
\end{proposition}

There are multiple equivalent ways of defining a linear category, some of which can be found in \cite{bierman1995categorical,mellies2003categorical,mellies2009categorical,schalk2004categorical}. For example, a linear category can be defined as a symmetric monoidal category equipped with a comonad whose coEilenberg-Moore category is a Cartesian category (i.e. a category with finite products) and such that the canonical adjunction between the base category and the coEilenberg-Moore category is a symmetric monoidal adjunction. Another way, which is of particular interest to this paper, is that in the presence of finite products, a monoidal coalgebra modality can be defined as a coalgebra modality that has the Seely isomorphisms. 

\begin{definition} \label{Seelydef}
 A coalgebra modality $(\oc, \rho, \varepsilon, \Delta, e)$ on a symmetric monoidal category $(\mathbb{X}, \otimes, R, \sigma)$ with finite products $\times$ and terminal object $\mathsf{T}$ is said to have the \textbf{Seely isomorphisms} \cite[Definion 10]{Blute2019} if the natural transformation $\chi_{A,B}: \oc(A \times B) \to \oc A \otimes B$ defined as:
\begin{equation}\label{}\begin{gathered} \chi_{A,B} := \xymatrixcolsep{5pc}\xymatrix{\oc(A \times B) \ar[r]^-{\Delta_{A \times B}} & \oc(A \times B) \otimes \oc(A \times B) \ar[r]^-{\oc(\pi_0) \otimes \oc(\pi_1)} & \oc A \otimes \oc B
 } \end{gathered}\end{equation}
is a natural isomorphism (where $\pi_0: A \times B \to A$ and $\pi_1: A \times B \to B$ are the projection maps of the product) and the map $\mathsf{e}_\mathsf{T}: \oc(\mathsf{T}) \to R$ is an isomorphism. A \textbf{monoidal storage category} \cite[Definition 3.1.4]{blute2015cartesian} is a symmetric monoidal category with finite products and a coalgebra modality which has the Seely isomorphisms. 
\end{definition}

Monoidal storage categories are also sometimes known as \textbf{new Seely categories} \cite{bierman1995categorical,mellies2003categorical}. 

\begin{theorem}\cite[Theorem 3.1.6]{blute2015cartesian} Every monoidal storage category is a linear category and conversely, every linear category with finite products is a monoidal storage category.  
\end{theorem} 

In particular, the above theorem implies that, in the presence of finite products, every coalgebra modality with the Seely isomorphisms is a monoidal coalgebra modality and conversely that every monoidal coalgebra modality has the Seely isomorphisms. To see how to construct one from the other see \cite[Section 3.1]{blute2015cartesian}. Once again, there is multiple equivalent ways of defining a monoidal storage category \cite{bierman1995categorical,mellies2003categorical,mellies2009categorical,schalk2004categorical}. For example, a monoidal storage category can be defined as a symmetric monoidal category with finite products equipped with a comonad whose coKleisli category is a Cartesian category and such that the canonical adjunction between the base category and the coKleisli category is a symmetric monoidal adjunction. With this in mind, the above theorem may be derived by considering the fact that the canonical adjunction for the coKleisli category factors through the coEilenberg-Moore category. In Sections \ref{polysec}, \ref{bisec}, and \ref{consec} we will explain why each coalgebra modality has the Seely isomorphisms and is therefore also a monoidal coalgebra modality.

We now turn our attention back to differential categories: 

\begin{definition}\label{difflindef} A \textbf{differential linear category} is a differential category whose coalgebra modality is a monoidal coalgebra modality. A \textbf{differential storage category} is a differential linear category with finite products. 
\end{definition}

The definition of a differential linear category might seem a bit lacking. Indeed, one might expect some compatibility coherences between the deriving transformation $\mathsf{d}$ and the symmetric monoidal endofunctor structure. However, this comes for free and said coherence is known as the \textbf{monoidal rule} \cite[Theorem 4]{Blute2019}. It is also worth mentioning that the differential structure of a differential linear category can be equivalent be axiomatized by a natural transformation ${\eta_A: A \to \oc A}$ known as the \textbf{codereliction} \cite{blute2006differential,Blute2019,fiore2007differential}. That said, it is the deriving transformation that plays the more important role when discussing integration and antiderivatives (though of course, the deriving transformation is built from a codereliction in a differential linear category \cite[Theorem 4]{Blute2019}). 

We conclude this section with a useful property of $\mathsf{K}$ and $\mathsf{J}$ for differential linear categories. 

\begin{proposition}\label{KJm} \cite[Proposition 4.5]{cockett_lemay_2018} In a differential linear category, 
\begin{enumerate}[{\em (i)}]
\item $\mathsf{K}$ satisfies the following equality: 
\begin{align*} 
(\mathsf{K}_A \otimes 1_{\oc B}) \mathsf{m}_{A, B} = \mathsf{m}_{A, B} \mathsf{K}_{A \otimes B} = (1_{\oc A} \otimes \mathsf{K}_B)\mathsf{m}_{A, B} 
\end{align*}
\item $\mathsf{J}$ satisfies the following equality: 
\begin{align*}  
(\mathsf{J}_A \otimes 1_{\oc B}) \mathsf{m}_{A, B} = \mathsf{m}_{A, B} \mathsf{J}_{A \otimes B} = (1_{\oc A} \otimes \mathsf{J}_B)\mathsf{m}_{A, B} 
\end{align*}
\end{enumerate}
\end{proposition}


\section{Differential Categories with Antiderivatives}\label{diffantisec}

In classical single-variable calculus, differentiation and integration are related by the two Fundamental Theorems of Calculus. Differential categories with antiderivatives were introduced to study and interpret integration and the Fundamental Theorems of Calculus in the differential category setting. In this section we give a brief overview of differential categories with \emph{antiderivatives} (Definition \ref{antidef}) and discuss and certain important consequences of having antiderivatives (Proposition \ref{FT2prop}). For more details on the story of integration and antiderivatives, we refer the reader to \cite{cockett_lemay_2018,ehrhard2017introduction}. 

\begin{definition}\label{antidef} A differential categories is said to have \textbf{antiderivatives} \cite[Definition 6.1]{cockett_lemay_2018} if $\mathsf{K}$ is a natural isomorphism. 
\end{definition}

Ehrhard's definition of antiderivatives \cite[Section 3.2]{ehrhard2017introduction} was that $\mathsf{J}$ be a natural transformation instead. While this was sufficient to construct an integral and prove Poincar\'e's Lemma \cite[Proposition 13]{ehrhard2017introduction}, one does not necessarily obtain the Second Fundamental Theorem of Calculus for free. On the other hand, $\mathsf{K}$ allows one to construct an integral which satisfies both the Poincar\'e's Lemma as well as the Second Fundamental Theorem of Calculus (Proposition \ref{FT2prop}), and also implies that $\mathsf{J}$ is a natural isomorphism. To obtain the Second Fundamental Theorem of Calculus from $\mathsf{J}$ being a natural isomorphism, this required an extra assumption about the deriving transformation known as the \emph{Taylor} property. In fact, this Taylor property provides an equivalence between the two definitions of antiderivatives. 

\begin{proposition} \cite[Proposition 6.1]{cockett_lemay_2018} For a differential category, the following are equivalent: 
\begin{enumerate}[{\em (i)}]
\item $\mathsf{K}$ is a natural isomorphism; 
\item $\mathsf{J}$ is a natural isomorphism and the deriving transformation $\mathsf{d}$ is \textbf{Taylor} \cite[Definition 5.3]{cockett_lemay_2018}, that is, if for maps $f: \oc A \to B$ and $g: \oc A \to B$ such that $\mathsf{d}_A f = \mathsf{d}_A g$, then:
\[f + \oc(0)g= g + \oc(0)f\] 
\end{enumerate}
\end{proposition}

To provide some intuition about the Taylor property, first note that precomposing with $\oc(0)$ is to be thought of as evaluating a smooth map $f: \oc A \to B$ at $0$, which therefore results in a constant function. That the deriving transformation is Taylor says that two smooth maps with the same derivative differ simply by a constant. 

In a differential category with antiderivatives, the integral is constructed as follows (which is equal to the integral constructed in \cite{ehrhard2017introduction}): 

\begin{definition} \label{ASdef} In a differential category with antiderivatives, the \textbf{antiderivative integral transformation} \cite[Definition 6.2]{cockett_lemay_2018} is the natural transformation $\mathsf{s}_A: \oc A \to \oc A \otimes A$ defined as follows: 
\begin{equation}\label{Sdef}\begin{gathered} \mathsf{s}_A := \xymatrixcolsep{5pc}\xymatrix{ \oc A \ar[r]^-{\mathsf{K}^{-1}_A} & \oc A \ar[r]^-{\mathsf{d}^\circ_A} & \oc A \otimes A} \end{gathered}\end{equation}
\end{definition}

Similar to the deriving transformation, the antiderivative integral transformation satisfies the basic axioms of integration from classical calculus such as the Rota-Baxter rule \cite{guo2012introduction}, the integral of a constant function is a linear function, and polynomial integration. The antiderivative integral transformation is an example of the more general concept of an integral transformation \cite[Definition 3.4]{cockett_lemay_2018} which axiomatizes integration separate from differentiation. In particular, one can integrate maps of type $f: \oc A \otimes A \to B$, where the integral is the smooth map $\mathsf{S}[f]: \oc A \to B$ defined as the composite $\mathsf{S}[f] : = \mathsf{s}_A f$. For more intuition on how to interpret this integral and examples of differential categories with antiderivatives, see \cite{cockett_lemay_2018,ehrhard2017introduction}. 

Here is a list of important coherences between the differential and integral structure of a differential category with antiderivatives: 

\begin{proposition}\label{FT2prop} In a differential category with antiderivatives:
\begin{enumerate}[{\em (i)}]
\item The antiderivative integral transformation satisfies the following equality:
\[\mathsf{s}_A = \mathsf{d}^\circ_A(\mathsf{J}^{-1}_A \otimes 1_A)\] 
\item The deriving transformation $\mathsf{d}$ and the antiderivative integral transformation $\mathsf{s}$ satisfy the \textbf{Second Fundamental Theorem of Calculus} \cite[Definition 5.1]{cockett_lemay_2018}, that is, the following equality holds: 
\begin{equation}\label{FT2}\begin{gathered} \mathsf{s}_A \mathsf{d}_A + \oc(0) = 1 \end{gathered}\end{equation}
\item The deriving transformation $\mathsf{d}$ and the antiderivative integral transformation $\mathsf{s}$ satisfy the \textbf{Poincar\'e Condition} \cite[Definition 5.5]{cockett_lemay_2018}, that is, if a map $f: \oc A \otimes A \to B$ satisfies the following equality: 
\[(\mathsf{d}_A \otimes 1_A)f= (1_{\oc A} \otimes \sigma_{A,A})(\mathsf{d}_A \otimes 1_A)f\]
then $\mathsf{d}_A\mathsf{s}_A f = f$. 
\item For the monoidal unit $R$, the deriving transformation $\mathsf{d}_R: \oc R \to \oc R$ and the antiderivative integral transformation $\mathsf{s}_R: \oc R \to \oc R$ satisfy the \textbf{First Fundamental Theorem of Calculus} \cite[Definition 5.7]{cockett_lemay_2018}, that is, the following equality holds: 
\begin{equation}\label{FT1}\begin{gathered} \mathsf{d}_R\mathsf{s}_R = 1_{\oc R}\end{gathered}\end{equation}
\end{enumerate}
\end{proposition} 

It might be useful to provide a bit of intuition here (for a more detailed explanation we again refer the reader to \cite{cockett_lemay_2018,ehrhard2017introduction}). Recall that the Second Fundamental Theorem of Calculus (in the one variable case) states that the integral of the derivative of a function on a closed interval is equal to the difference of at the end points:
\begin{align*}
\int_a^b\frac{{\sf d} f(t)}{{\sf d}t}(s)~{\sf d}s=f(b)-f(a)
\end{align*}
In a differential category with antiderivatives, every smooth map $f: \oc A \to B$ satisfies the Second Fundamental Theorem of Calculus in the sense that $\mathsf{S}\left[ \mathsf{D}[f] \right] + \oc(0) f = f$. Naively, using notation of single-variable calculus, this last identity should be interpreted as follows: 
\[\int_0^x\frac{{\sf d} f(t)}{{\sf d}t}(s)~{\sf d}s + f(0)=f(x)\]
where we had to do some rearranging since we do not necessarily have negatives. On the other hand, recall that in single-variable calculus, the First Fundamental Theorem of Calculus states that the derivative of the integral of a function is equal to the original function: 
\begin{align*}
\frac{{\sf d} (\int_a^t f(u)~{\sf d}u)}{{\sf d}t}(x)=f(x)
\end{align*}
In differential category with antiderivatives, the First Fundamental Theorem of Calculus does not hold in the sense that $\mathsf{d}\mathsf{s}=1$. Instead the Poincar\'e Condition gives necessary and sufficient conditions for a map $f: \oc A \otimes A \to B$ to satisfy the First Fundamental Theorem of Calculus in the sense that $\mathsf{D}\left[ \mathsf{S}[f] \right] = f$. However, a special case is the monoidal unit, where every coKleisli map $f: \oc R \to B$ satisfies both Fundamental Theorems of Calculus -- making the monoidal unit a \textbf{calculus object} \cite[Definition 5.7]{cockett_lemay_2018}. Finally, every differential category with antiderivatives is a \textbf{calculus category} \cite[Definition 5.6]{cockett_lemay_2018} -- which axiomatizes the compatible relation between differentiation and integration via the Fundamental Theorems of Calculus. 

\section{Differential Linear Categories with Antiderivatives}\label{difflinsec}

In this section, we turn our attention to studying when a differential linear category has antiderivatives. In particular, we will prove Theorem \ref{antithm} which provides necessary and sufficient conditions for when a differential linear category has antiderivatives. Briefly, for a differential linear category to have integration and antiderivatives, it is sufficient to know how to integrate over the monoidal unit $R$ and also that the Second Fundamental Theorem of Calculus holds. This observation will greatly simplify showing that the differential linear categories of Sections \ref{polysec}, \ref{bisec}, and \ref{consec} have antiderivatives. 

It may be useful to first provide an outline of how we will obtain our desired result. While it is possible to provide a direct proof that $\mathsf{K}$ is a natural isomorphism, the direct calculation is somewhat tedious as it amounts simply too long strings of equations. Therefore, we will provide an alternative proof with smaller intermediate steps, which for the reader is hopefully more informative and enjoyable to read. We will start by observing that a differential linear category has antiderivatives if and only if $\mathsf{K}_R$ is an isomorphism (Proposition \ref{JKprop1}). Then we will assume that we have an integration map for the monoidal unit $\mathsf{s}_R$ which is compatible with $\mathsf{d}_R$ in that the Second Fundamental Theorem of Calculus holds (Definition \ref{f2def}). Our objective will then be to construct $\mathsf{K}^{-1}_R$ using $\mathsf{s}_R$. To do so, we will first show that $\mathsf{J}_R$ is an isomorphism (Lemma \ref{Jinv}). Then we will construct $\mathsf{K}^{-1}_R$ using $\mathsf{s}_R$, $\mathsf{d}_R$, and $\mathsf{J}^{-1}_R$ (Lemma \ref{Kinv}). From here, we will be able to easily prove Theorem \ref{antithm} and conclude that $\mathsf{K}$ is a natural isomorphism. As a consequence, we can construct $\mathsf{K}^{-1}$, $\mathsf{J}^{-1}$, and the antiderivative integral transformation $\mathsf{s}$ in terms of $\mathsf{s}_R$. 

We begin, as promised, with the observation that for a differential linear category, having antiderivatives is completely determined by the monoidal unit component of $\mathsf{K}$. 

\begin{proposition} \label{JKprop1} A differential linear category has antiderivatives if and only if for the monoidal unit $R$, $\mathsf{K}_R$ is an isomorphism. 
\end{proposition}
\begin{proof} Suppose that $\mathsf{K}$ is a natural isomorphism. Then by definition, $\mathsf{K}_R$ is an isomorphism. Conversely, suppose that $\mathsf{K}_R$ is an isomorphism. Define $\mathsf{K}^{-1}_A: \oc A \to \oc A$ as follows: 
\[\mathsf{K}^{-1}_A := \xymatrixcolsep{5pc}\xymatrix{ \oc A \ar[r]^-{\mathsf{m}_R \otimes 1_{\oc A}} & \oc R \otimes \oc A \ar[r]^-{\mathsf{K}^{-1}_R \otimes 1_{\oc A}} & \oc R \otimes \oc A \ar[r]^-{\mathsf{m}_{R, A}} & \oc A } \]
Then we have that: 
\begin{align*}
\mathsf{K}^{-1}_A \mathsf{K}_A &=~ (\mathsf{m}_R \otimes 1_{\oc A} ) (\mathsf{K}^{-1}_R \otimes 1_{\oc A}) \mathsf{m}_{R, A} \mathsf{K}_A \\
&=~ (\mathsf{m}_R \otimes 1_{\oc A})(\mathsf{K}^{-1}_R \otimes 1_{\oc A}) (\mathsf{K}_R \otimes 1_{\oc A}) \mathsf{m}_{R, A} \tag{Prop \ref{KJm}.i} \\
&=~ (\mathsf{m}_R \otimes 1_{\oc A}) \mathsf{m}_{R, A} \\
&=~ 1_{\oc A} \tag{Def \ref{moncoalgmod}.i} \\ \\
\mathsf{K}_A \mathsf{K}^{-1}_A &=~ \mathsf{K}_A (\mathsf{m}_R \otimes 1_{\oc A} ) (\mathsf{K}^{-1}_R \otimes 1_{\oc A}) \mathsf{m}_{R, A} \\
&=~ (\mathsf{m}_R \otimes 1_{\oc A} ) (\mathsf{K}^{-1}_R \otimes 1_{\oc A}) (1_{\oc R} \otimes \mathsf{K}_A) \mathsf{m}_{R, A} \\
&=~ (\mathsf{m}_R \otimes 1_{\oc A})(\mathsf{K}^{-1}_R \otimes 1_{\oc A}) (\mathsf{K}_R \otimes 1_{\oc A}) \mathsf{m}_{R, A} \tag{Prop \ref{KJm}.i} \\
&=~ (\mathsf{m}_R \otimes 1_{\oc A}) \mathsf{m}_{R, A} \\
&=~ 1_{\oc A} \tag{Def \ref{moncoalgmod}.i} 
\end{align*}
So we conclude that $\mathsf{K}$ is a natural isomorphism. 
\end{proof}

Before we start working with integration, we consider the following useful observations about the monoidal unit components of $\mathsf{K}$ and $\mathsf{J}$: 

\begin{lemma}\label{KJcR} In a differential category, the following equalities hold for the monoidal unit $R$: 
\begin{enumerate}[{\em (i)}]
\item $ \mathsf{K}_R\mathsf{d}^\circ_R = \mathsf{d}^\circ_R \mathsf{J}_R$
\item $\mathsf{d}_R\mathsf{K}_R = \mathsf{J}_R\mathsf{d}_R$ 
\item $ \mathsf{d}_R \mathsf{d}^\circ_R = \mathsf{J}_R$ 
\end{enumerate}
\end{lemma}
\begin{proof} If the identities involving the monoidal unit look a bit off, recall that we are working in a \emph{strict} monoidal category and so $R \otimes R = R$. Therefore, $(i)$ and $(ii)$ are simply re-expressions of Proposition \ref{KJp} with the strict monoidal structure in mind. While $(iii)$ is re-expressing Proposition \ref{Wp}, using that $\sigma_{R,R} = 1_R$. 
\end{proof}

We turn our attention now to working with integration on the monoidal unit $R$. Integration will be captured by a map of type $\mathsf{s}_R: \oc R \to \oc R$ which is compatible with the deriving transformation $\mathsf{d}_R: \oc R \to \oc R$ in the sense that the Second Fundamental Theorem of Calculus as in Proposition \ref{FT2prop} holds. The idea here is that $\mathsf{s}_R$ is the antiderivative integral transformation (Definition \ref{ASdef}) at the monoidal unit. Once again, if the types of $\mathsf{s}_R$ and $\mathsf{d}_R$ look a bit off, recall that we are working in a \emph{strict} monoidal category and so $\oc R \otimes R = \oc R$. 

\begin{definition}\label{f2def} In a differential category, for the monoidal unit $R$, a map $\mathsf{s}_R: \oc R \to \oc R$ satisfies \textbf{[ftc.2]} if the following equality holds:
\begin{equation}\label{f2}\begin{gathered} \mathsf{s}_R \mathsf{d}_R + \oc(0) = 1_{\oc R} \end{gathered}\end{equation}
\end{definition}

\begin{lemma}\label{sJlem} In a differential category, if a map $\mathsf{s}_R: \oc R \to \oc R$ satisfies \textbf{[ftc.2]}, then the following equality holds: 
\begin{equation}\label{sJ}\begin{gathered} \mathsf{s}_R\mathsf{J}_R = \mathsf{d}^\circ_R \end{gathered}\end{equation}
\end{lemma}
\begin{proof} We prove the identity by the following calculation: 
\begin{align*}
\mathsf{s}_R \mathsf{J}_R &=~ \mathsf{s}_R \mathsf{d}_R \mathsf{d}^\circ_R \tag{Lem. \ref{KJcR}.iii} \\
&=~ \mathsf{s}_R \mathsf{d}_R \mathsf{d}^\circ_R + 0 \\
&=~ \mathsf{s}_R \mathsf{d}_R \mathsf{d}^\circ_R + \mathsf{d}^\circ_R (\oc(0) \otimes 0) \\ 
&=~ \mathsf{s}_R \mathsf{d}_R \mathsf{d}^\circ_R + \oc(0)\mathsf{d}^\circ_R \tag{Nat. $\mathsf{d}^\circ$}\\
&=~ (\mathsf{s}_R \mathsf{d}_R + \oc(0) )\mathsf{d}^\circ_R \\
&=~ \mathsf{d}^\circ_R \tag{Def. \ref{f2def}} 
\end{align*}
\end{proof} 

\begin{corollary}\label{cor500} In a differential category such that $\mathsf{J}_R$ is an isomorphism, define the map ${\mathsf{s}_R: \oc R \to \oc R}$ as follows: 
\[\mathsf{s}_R := \xymatrixcolsep{5pc}\xymatrix{ \oc R \ar[r]^-{\mathsf{d}^\circ_R} & \oc R \ar[r]^-{\mathsf{J}^{-1}_R} & \oc R}\] 
If $\mathsf{s}_R$ satisfies \textbf{[ftc.2]} then it is the unique map which satisfies \textbf{[ftc.2]}. 
\end{corollary}
\begin{proof} Suppose that $\mathsf{s}_R$ satisfies \textbf{[ftc.2]} and that there is another map $\mathsf{s}^\prime_R$ which satisfies \textbf{[ftc.2]}. Then we have that: 
\begin{align*} \mathsf{s}^\prime_R &=~ \mathsf{s}^\prime_R \mathsf{J}_R\mathsf{J}^{-1}_R \\
&=~ \mathsf{d}^\circ_R \mathsf{J}^{-1}_R \tag{Lem \ref{sJlem}} \\
&=~ \mathsf{s}_R
\end{align*}
So we conclude that $\mathsf{s}_R$ is the unique map which satisfies \textbf{[ftc.2]}. 
\end{proof}

Note that in the setting of the above corollary, uniqueness justifies the use of $\mathsf{s}_R$ as appropriate notation. It is important to note that $\mathsf{J}_R$ being an isomorphism does not necessarily imply that $\mathsf{d}^\circ_R\mathsf{J}^{-1}_R$ satisfies \textbf{[ftc.2]}. Next, we show that in the case of a differential linear category, having such an $\mathsf{s}_R$ implies that $\mathsf{J}_R$ is an isomorphism. 

\begin{lemma}\label{Jinv} In a differential linear category, if a map $\mathsf{s}_R: \oc R \to \oc R$ satisfies \textbf{[ftc.2]}, then $\mathsf{J}_R$ is an isomorphism with inverse $\mathsf{J}^{-1}_R$ defined as follows: 
\[\mathsf{J}^{-1}_R := \xymatrixcolsep{5pc}\xymatrix{ \oc R \ar[r]^-{\mathsf{m}_R \otimes 1_{\oc R}} & \oc R \otimes \oc R \ar[r]^-{\mathsf{s}_R \otimes 1_{\oc R}} & \oc R \otimes \oc R \ar[r]^-{\mathsf{m}_{R, R}} & \oc R } \]
and furthermore the following equality holds:
\begin{equation}\label{sJ2}\begin{gathered} \mathsf{s}_R = \mathsf{d}^\circ_R\mathsf{J}^{-1}_R \end{gathered}\end{equation} 
\end{lemma}
\begin{proof} We prove that $\mathsf{J}^{-1}_R$ is the inverse of $\mathsf{J}_R$ by the following calculations: 
\begin{align*}
\mathsf{J}^{-1}_R \mathsf{J}_R &=~ (\mathsf{m}_R \otimes 1_{\oc R}) (\mathsf{s}_R \otimes 1_{\oc R}) \mathsf{m}_{R, R} \mathsf{J}_R \\
&=~ (\mathsf{m}_R \otimes 1_{\oc R}) (\mathsf{s}_R \otimes 1_{\oc R})(\mathsf{J}_R \otimes 1_{\oc R}) \mathsf{m}_{R, R} \tag{Prop \ref{KJm}.i} \\
&=~ (\mathsf{m}_R \otimes 1_{\oc R}) (\mathsf{d}^\circ_R \otimes 1_{\oc R}) \mathsf{m}_{R, R} \tag{Prop \ref{sJlem}} \\
&=~ (\mathsf{m}_R \otimes 1_{\oc R}) \mathsf{m}_{R, R} \tag{Prop \ref{dcircm}} \\
&=~ 1_{\oc R} \tag{Def \ref{moncoalgmod}.i} \\ \\
\mathsf{J}_R \mathsf{J}^{-1}_R &=~ \mathsf{J}_R (\mathsf{m}_R \otimes 1_{\oc R}) (\mathsf{s}_R \otimes 1_{\oc R}) \mathsf{m}_{R, R} \\
&=~ (\mathsf{m}_R \otimes 1_{\oc R}) (\mathsf{s}_R \otimes 1_{\oc R}) (1_{\oc R} \otimes \mathsf{J}_R) \mathsf{m}_{R, R} \\
&=~ (\mathsf{m}_R \otimes 1_{\oc R}) (\mathsf{d}^\circ_R \otimes 1_{\oc R}) \mathsf{m}_{R, R} \tag{Prop \ref{sJlem}} \\
&=~ (\mathsf{m}_R \otimes 1_{\oc R}) \mathsf{m}_{R, R} \tag{Prop \ref{dcircm}} \\
&=~ 1_{\oc R} \tag{Def \ref{moncoalgmod}.i} 
\end{align*}
So we conclude that $\mathsf{J}_R$ is an isomorphism. By Corollary \ref{cor500}, this implies that $\mathsf{s}_R = \mathsf{d}^\circ_R\mathsf{J}^{-1}_R$. 
\end{proof} 

It is worth pointing out that the formula for $\mathsf{J}^{-1}_R$ in Lemma \ref{Jinv} does not come from out of the blue. Indeed, the construction of $\mathsf{J}^{-1}_R$ is a specialization of the construction of $\mathsf{J}^{-1}$ from an integral transformation $\mathsf{s}$ as found in \cite[Theorem 3]{cockett_lemay_2018}. We are now in a position to construct $\mathsf{K}^{-1}_R$ using using $\mathsf{s}_R$, $\mathsf{d}_R$, and $\mathsf{J}^{-1}_R$. Once again, it is worth pointing out that the construction of $\mathsf{K}^{-1}_R$ in the lemma below is not as random as it appears. The construction of $\mathsf{K}^{-1}_R$ is a re-expression of the construction of $\mathsf{K}^{-1}$ using $\mathsf{J}^{-1}$ as found in \cite[Proposition 20]{cockett_lemay_2018}. 

\begin{lemma}\label{Kinv} In a differential linear category, if a map $\mathsf{s}_R: \oc R \to \oc R$ satisfies \textbf{[ftc.2]}, then $\mathsf{K}_R$ is an isomorphism with inverse $\mathsf{K}^{-1}_R$ defined as follows: 
 \[ \mathsf{K}^{-1}_R := \left( \xymatrixcolsep{3.65pc} \xymatrixrowsep{0.75pc}\xymatrix{\oc R \ar[r]^-{\mathsf{s}_R} & \oc R \ar[r]^-{\mathsf{J}^{-1}_R} & \oc R \ar[r]^-{\mathsf{d}_{R}} & \oc R } \right)+ 
\xymatrixcolsep{3.65pc} \left(\xymatrix{\oc R \ar[r]^-{\oc(0)} & \oc R }\right)\] 
and furthermore the following equality holds:
\begin{equation}\label{sJ3}\begin{gathered} \mathsf{s}_R = \mathsf{K}^{-1}_R\mathsf{d}^\circ_R \end{gathered}\end{equation} 
\end{lemma}
\begin{proof} By Lemma \ref{Jinv}, we know that $\mathsf{J}_R$ is an isomorphism and that $\mathsf{s}_R = \mathsf{d}^\circ_R\mathsf{J}^{-1}_R$. We prove that $\mathsf{K}^{-1}_R$ is the inverse of $\mathsf{K}_R$ by the following calculations: 
\begin{align*}
\mathsf{K}^{-1}_R \mathsf{K}_R &=~ \left(\mathsf{s}_R\mathsf{J}^{-1}_R\mathsf{d}_R + \oc(0) \right) \mathsf{K}_R \\ 
&=~ \mathsf{s}_R\mathsf{J}^{-1}_R\mathsf{d}_R\mathsf{K}_R + \oc(0) \mathsf{K}_R \\
&=~ \mathsf{s}_R\mathsf{J}^{-1}_R\mathsf{J}_R\mathsf{d}_R + \oc(0) \tag{Lem \ref{KJCR}.ii + Prop \ref{KJp}} \\
&=~ \mathsf{s}_R\mathsf{d}_R + \oc(0) \\
&=~ 1_{\oc R} \tag{Def. \ref{f2def}} \\ \\
 \mathsf{K}_R\mathsf{K}^{-1}_R &=~ \mathsf{K}_R \left(\mathsf{s}_R\mathsf{J}^{-1}_R\mathsf{d}_R + \oc(0) \right) \\
 &=~ \mathsf{K}_R \mathsf{s}_R\mathsf{J}^{-1}_R\mathsf{d}_R + \mathsf{K}_R\oc(0) \\
 &=~ \mathsf{K}_R\mathsf{d}^\circ_R\mathsf{J}^{-1}_R \mathsf{J}^{-1}_R\mathsf{d}_R + \oc(0) \tag{Lem \ref{Jinv} + Prop \ref{KJp}} \\
 &=~ \mathsf{d}^\circ_R\mathsf{J}_R\mathsf{J}^{-1}_R \mathsf{J}^{-1}_R\mathsf{d}_R + \oc(0) \tag{Lem \ref{KJCR}.i} \\
 &=~ \mathsf{d}^\circ_R \mathsf{J}^{-1}_R\mathsf{d}_R + \oc(0) \\
&=~ \mathsf{s}_R\mathsf{d}_R + \oc(0) \tag{Lem \ref{Jinv}} \\
&=~ 1_{\oc R} \tag{Def. \ref{f2def}}
\end{align*}
So we conclude that $\mathsf{K}_R$ is an isomorphism. We compute the other identity as follows:
\begin{align*} \mathsf{s}_R &=~ \mathsf{d}^\circ_R \mathsf{J}^{-1}_R \tag{Lem \ref{Jinv}} \\
&=~ \mathsf{K}^{-1}_R \mathsf{K}_R \mathsf{d}^\circ_R \mathsf{J}^{-1}_R \\
&=~ \mathsf{K}^{-1}_R \mathsf{d}^\circ_R \mathsf{J}_R\mathsf{J}^{-1}_R\tag{Lem \ref{KJcR}.i} \\
&=~ \mathsf{K}^{-1}_R\mathsf{d}^\circ_R 
\end{align*}
\end{proof} 

We may now easily prove the main technical result of this paper.

\begin{theorem}\label{antithm} A differential linear category has antiderivatives if and only if for the monoidal unit $R$ there is a map $\mathsf{s}_R: \oc R \to \oc R$ which satisfies \textbf{[ftc.2]}. \end{theorem} 
\begin{proof} Suppose that $\mathsf{K}$ is a natural isomorphism. Consider the component of the antiderivative integral transformation (Definition \ref{ASdef}) $\mathsf{s} = \mathsf{K} \mathsf{d}^\circ$ at the monoidal unit $\mathsf{s}_R: \oc R \to \oc R$. By Proposition \ref{FT2prop}.$(iii)$, $\mathsf{s}_R$ satisfies \textbf{[ftc.2]}. Conversely, suppose that we have a map $\mathsf{s}_R: \oc R \to \oc R$ which satisfies \textbf{[ftc.2]}. Then by Lemma \ref{Kinv}, it follows that $\mathsf{K}_R$ is an isomorphism. And therefore by Proposition \ref{JKprop1} we conclude that $\mathsf{K}$ is a natural isomorphism. 
\end{proof}

At this point, it may be worth briefly discussing the practical differences between Theorem \ref{antithm} and Proposition \ref{JKprop1}. From a purely computational point of view, when working with $\mathsf{s}_R$ one needs to only check one identity (i.e. $\mathsf{s}_R\mathsf{d}_R + \oc(0) = 1_{\oc R}$) compared to the two identities one needs to check for $\mathsf{K}^{-1}_R$ (i.e. $\mathsf{K}^{-1}_R \mathsf{K}_R = 1_{\oc R}$ and $ \mathsf{K}_R\mathsf{K}^{-1}_R=1_{\oc R}$). In practice, working with and computing $\mathsf{K}$ and $\mathsf{J}$ and their inverses may not be simple or obvious, see for example the complex formula for $\mathsf{K}^{-1}$ for real smooth functions in \cite[Proposition 6.1]{cruttwell2019integral}. On the other hand, working with $\mathsf{s}$ is much more intuitive as it is, in general, the expected line integration operator. In particular, integration is even simpler for the monoidal unit as it amounts to integration in one variable which is very well behaved and easy to work with.  

We conclude this section by expressing $\mathsf{K}^{-1}$, $\mathsf{J}^{-1}$, and $\mathsf{s}$ in terms of $\mathsf{s}_R$.  

\begin{corollary}\label{KJcor} In a differential linear category with antiderivatives, the following equalities holds: 
\begin{enumerate}[{\em (i)}]
\item $\mathsf{K}_A^{-1}= (\mathsf{m}_R \otimes \mathsf{m}_R \otimes 1_{\oc A})(\mathsf{s}_R \otimes \mathsf{s}_R \otimes 1_{\oc A})(\mathsf{m}_{R,R} \otimes 1_{\oc A})(\mathsf{d}_R \otimes 1_{\oc A})\mathsf{m}_{R,A} + \oc(0)$
 \[ \mathsf{K}^{-1}_A := \left( \xymatrixcolsep{3pc} \xymatrixrowsep{0.75pc}\xymatrix{\oc A \ar[r]^-{\mathsf{m}_R \otimes \mathsf{m}_R \otimes 1_{\oc A}} & \oc R \otimes \oc R \otimes \oc A \ar[r]^-{\mathsf{s}_R \otimes \mathsf{s}_R \otimes 1_{\oc A}} & \\
 \oc R \otimes \oc R \otimes \oc A \ar[r]^-{\mathsf{m}_{R,R} \otimes 1_{\oc A}} & \oc R \otimes \oc A \ar[r]^-{\mathsf{d}_R \otimes 1_{\oc A}} & \oc R \otimes \oc A \ar[r]^-{\mathsf{m}_{R,A}} & \oc A } \right)+ 
\xymatrixcolsep{2.5pc} \left(\xymatrix{\oc A \ar[r]^-{\oc(0)} & \oc A }\right)\] 
\item $\mathsf{J}^{-1}_A = (\mathsf{m}_R \otimes 1_{\oc A})(\mathsf{s}_R \otimes 1_{\oc A})\mathsf{m}_{R,A}$
 \[ \mathsf{J}^{-1}_A := \xymatrixcolsep{5pc}\xymatrix{\oc A \ar[r]^-{\mathsf{m}_R \otimes 1_{\oc A}} & \oc R \otimes \oc A\ar[r]^-{\mathsf{s}_R \otimes 1_{\oc A}} & \oc R \otimes \oc A \ar[r]^-{\mathsf{m}_{R,A}} & \oc A  
 } \]
\item $\mathsf{s}_A = (\mathsf{m}_R \otimes 1_A)(\mathsf{s}_R \otimes \mathsf{d}^\circ)(\mathsf{m}_{R,A} \otimes 1_A)$
 \[ \mathsf{s}_A := \xymatrixcolsep{5pc}\xymatrix{\oc A \ar[r]^-{\mathsf{m}_R \otimes 1_{\oc A}} & \oc R \otimes \oc A \ar[r]^-{\mathsf{s}_R \otimes \mathsf{d}^\circ_A} & \oc R \otimes \oc A \otimes A \ar[r]^-{\mathsf{m}_{R,A} \otimes 1_{\oc A}} & \oc A  
 } \]
\end{enumerate}
\end{corollary}

\section{Polynomials}\label{polysec}

Before working with convenient vector spaces in Section \ref{consec}, it might be useful to work with a simpler example. In this section, we will briefly review one of the most well-known examples of a differential category, or rather, of a codifferential category (the dual of a differential category). This example is induced by the free symmetric algebra construction \cite[Section 8, Chapter XVI ]{lang2002algebra} and the differential structure corresponds to polynomial differentiation. This differential category was introduced in \cite{blute2006differential}, and in certain circumstances was also shown to have antiderivatives in \cite{cockett_lemay_2018}. While we do not go into full details, we will take advantage of Theorem \ref{antithm} and focus mostly on the monoidal unit. 

Let $R$ be a commutative semiring and $\mathsf{MOD}_R$ the category of $R$-modules and $R$-linear maps between them. We briefly explain how $\mathsf{MOD}^{op}_R$ is a differential linear category. $\mathsf{MOD}_R$ is an additive symmetric monoidal category with the standard tensor product and additive enrichment of $R$-modules. For an $R$-module $M$, the free commutative $R$-algebra over $M$ is known as the the free symmetric algebra over $M$ and is denoted by $\mathsf{Sym}(M)$. By the universal property of the free symmetric algebra, we obtain a monad $\mathsf{Sym}$ on $\mathsf{MOD}_R$ which is also an algebra modality which has the Seely isomorphisms: 
\begin{align*}\mathsf{Sym}(M \times N) \cong \mathsf{Sym}(M) \otimes \mathsf{Sym}(N) && \mathsf{Sym}(0)\cong R
\end{align*}
Therefore, $\mathsf{Sym}$ is a comonoidal algebra modality, that is, $\mathsf{Sym}$ is a monoidal coalgebra modality on $\mathsf{MOD}^{op}_R$. Furthermore, $\mathsf{Sym}$ comes equipped with (the dual of) a deriving transformation ${\mathsf{d}_M: \mathsf{Sym}(M) \to \mathsf{Sym}(M) \otimes M}$ given by multivariable polynomial differentiation. Therefore, $\mathsf{MOD}^{op}_R$ is a differential linear category (see \cite{Blute2019,blute2006differential} for more details). 

In particular for the monoidal unit, which is simply $R$ itself, $\mathsf{Sym}(R)$ is isomorphic as $R$-algebras to the polynomial ring $R[x]$. As a result, by abusing notation slightly, the deriving transformation can be interpreted as ${\mathsf{d}_R: R[x] \to R[x]}$ and is given by the standard differentiation of polynomials: 
\[\mathsf{d}_R\left( \sum \limits^n_{k=0} r_k x^k \right) = \sum \limits^n_{k=1} (k \cdot r_{k}) x^{k-1} \]
where on the right hand side, $\cdot$ is the multiplication in $R$ and $k$ is interpreted as the element of $R$ which is the sum of the multiplicative unit $k$-times. Unsurprisingly, the desired integral $\mathsf{s}_R$ will be given by the standard integration of polynomials. For this, we need that all positive sums of the multiplicative unit of $R$ are invertible, or equivalently, that there exists a (unique) semiring morphism ${\mathbb{Q}_{\geq 0} \to R}$ (where $\mathbb{Q}_{\geq 0}$ is the semiring of non-negative rational numbers). So suppose that for each $k \in \mathbb{N}$, where $\mathbb{N}$ is the set of natural numbers, that $k \in R$ is invertible with inverse $k^{-1}$. Define ${\mathsf{s}_R: R[x] \to R[x]}$ using the standard formula for polynomial integration: 
\[\mathsf{s}_R\left(\sum \limits^n_{k=0} r_k x^k \right) = \sum \limits^n_{k=0} \left( (k+1)^{-1} \cdot r_{k} \right) x^{k+1} \]
On the other hand, $\mathsf{Sym}(0): R[x] \to R[x]$ is precisely evaluating a polynomial at zero, which amounts to giving the polynomial's constant term: 
\[\mathsf{Sym}(0)\left(\sum \limits^n_{k=0} r_k x^k \right) = r_0\]
One can then easily check that $\mathsf{d}_R$ and $\mathsf{s}_R$ satisfy the Second Fundamental Theorem of Calculus: 
\begin{align*}
\mathsf{s}_R \left(\mathsf{d}_R\left( \sum \limits^n_{k=0} r_k x^k \right) \right) + \mathsf{Sym}(0)\left(\sum \limits^n_{k=0} r_k x^k \right) &=~ \mathsf{s}_R \left( \sum \limits^n_{k=1} (k \cdot r_{k}) x^{k-1} \right) + r_0 \\
&=~ \sum \limits^n_{k=1} \left(k \cdot (k-1+1)^{-1} \cdot r_k\right) x^{k-1+1} + r_0 \\
&=~ \sum \limits^n_{k=1} \left(k \cdot k^{-1} \cdot r_k \right) x^{k-1+1} + r_0 \\
&=~ \sum \limits^n_{k=1} r_k x^k + r_0 \\
&=~ \sum \limits^n_{k=0} r_k x^k
\end{align*}
And so $\mathsf{d}_R\mathsf{s}_R + \oc(0) = 1_{\oc R}$. If this looks backwards, recall that $\mathsf{MOD}^{op}_R$ is the differential linear category. 

\begin{theorem} Let $R$ be a commutative semiring such that all positive sums of the multiplicative unit are invertible. Then $\mathsf{MOD}^{op}_R$ is a differential linear category with antiderivatives. 
\end{theorem}

The induced (dual of) antiderivative integral transformation $\mathsf{s}_M: \mathsf{Sym}(M) \otimes M \to \mathsf{Sym}(M)$ gives a special kind of multivariable polynomial integration which is described in \cite[Example 1]{cockett_lemay_2018}. In particular, this multivariable polynomial integration satisfies the Rota-Baxter rule and the Second Fundamental Theorem of Calculus for any number of finite variables. 

\section{Biproduct Completion of Complete Semirings} \label{bisec}

In this section, we will show that certain generalizations of the relational model give a differential category with antiderivatives. By generalizations of the relation model, we mean the biproduct completion of a complete semiring -- which as the name indicates, gives a generalization of the category of sets and relations, $\mathsf{REL}$. In fact, $\mathsf{REL}$ was one of the original examples of a differential category \cite{blute2006differential} and of a differential category with antiderivatives \cite{cockett_lemay_2018,ehrhard2017introduction}. For more details on generalizations of the relational model, we invite the reader to see \cite{laird2013weighted,lamarche1992quantitative,ong2017quantitative}. 

Briefly, recall that a \textbf{complete semiring} is a semiring where one can have sums indexed by arbitrary sets $I$, which we denote by $\sum \limits_{i \in I}$, such that these summation operations satisfy certain distributivity and partitions axioms (see \cite[Chapter 22]{golan2013semirings} for more details). Now let $R$ be a complete commutative semiring. Define the category $R^\Pi$ whose objects are sets $X$ and where a map from $X$ to $Y$ is a set function $f: X \times Y \to R$. Composition of maps $f: X \times Y \to R$ and $g: Y \times Z \to R$ is the map $fg: X \times Z \to R$ defined as follows: 
\[fg(x,z) := \sum \limits_{y \in Y} f(x, y) \cdot g(y,z) \]
where $\cdot$ is the multiplication in $R$. The identity is given by the Kronecker function $\delta: X \times X \to R$, which is defined as follows: 
 \[\delta(x,y) := \begin{cases} 0 & \text{ if } x \neq y \\
1 & \text{ if } x =y 
\end{cases}\]
For a bit more intuition, maps of $R^\Pi$ should be viewed as generalized $R$-matrices. Composition corresponds to matrix multiplication.While the identity is the diagonal matrix of $1$'s on the diagonal and zero everywhere else. For an explicit example, consider the two-element Boolean algebra \cite{givant2008introduction} $B = \lbrace 0, 1\rbrace$, which is a complete commutative semiring. In this case, $B^\Pi$ is isomorphic to $\mathsf{REL}$, since every map $f: X \times Y \to B$ can be equivalently be described as a subset of $X \times Y$, which is precisely a relation between $X$ and $Y$. 

$R^\Pi$ is the biproduct completion of $R$ viewed as a one object category. The biproduct of objects is given by the disjoint union of sets $\sqcup$ and the zero object is empty set $\emptyset$. As such, $R^\Pi$ is an additive category where the zero maps $0: X \times Y \to R$ simply map everything to $0$, while the sum of maps $f +g: X \times Y \to R$ is defined by pointwise addition: 
\[(f+g)(x,y) := f(x,y) + g(x,y) \] 
$R^\Pi$ is also a symmetric monoidal category where the monoidal unit is a chosen singleton $\lbrace \ast \rbrace$ and the tensor product of objects is given by the standard Cartesian product of sets $\times$. This structure makes $R^\Pi$ an additive symmetric monoidal category.  

$R^\Pi$ is also a differential linear category. For each set $X$, let $\oc X$ be the free commutative monoid over $X$. Elements of $\oc X$ are are finite bags (also known as multisets) of elements of $X$:
\[ \llbracket x_1, \hdots, x_n \vert x_i \in X \rrbracket \in \oc X\]
and including the empty bag $\llbracket \rrbracket$. In particular for the disjoint union of sets and empty set, we also have the following: 
\[\oc(X \sqcup Y) \cong \oc X \times \oc Y \quad \quad \quad \oc \emptyset \cong \lbrace \ast \rbrace\]
This gives a coalgebra modality which satisfies the Seely isomorphisms, and therefore provides a monoidal coalgebra modality on $R^\Pi$ (for a full description of this monoidal coalgebra modality see \cite{laird2013weighted,lamarche1992quantitative,ong2017quantitative}). This monoidal coalgebra modality is in fact a \textbf{free exponential modality} \cite{mellies2017explicit}, making $R^\Pi$ a \textbf{Lafont category} \cite{mellies2009categorical}. The deriving transformation $\mathsf{d}_X: (\oc X \times X) \times \oc X \to R$ is defined as putting single elements into bags: 
\[\mathsf{d}_X((\llbracket x_1, \hdots, x_n \rrbracket, x), \llbracket y_1, \hdots, y_m \rrbracket) = m \cdot \delta(\llbracket x_1, \hdots, x_n, x \rrbracket, \llbracket y_1, \hdots, y_m \rrbracket) \]
Multiplying by $m=n+1$ takes into account that if we were in the unordered case, there would be $n+1$ possible ways of putting an element into a bag of size $n$. Of course the $n+1$ factor disappears in the case that semiring is additively idempotent (i.e. $1+1 = 1$), such as the two-element Boolean algebra $B$. Which is why the $n+1$ factor does not appear in the differential structure of $\mathsf{REL}$ as described in \cite{blute2006differential}.  

Focusing on the monoidal unit $\lbrace \ast \rbrace$, $\oc \lbrace \ast \rbrace$ is isomorphic as a commutative monoid to the set of natural numbers $\mathbb{N}$. The deriving transformation, expressed as ${\mathsf{d}_{\lbrace \ast \rbrace}: \mathbb{N} \times \mathbb{N} \to R}$, is then:
\[\mathsf{d}_{\lbrace \ast \rbrace}(n,m) = m \cdot \delta(n+1, m)\]
Therefore, as in the previous section, we will need inverse of all positive sums of the multiplicative unit to define integration. So once again, assume that for each $n \in \mathbb{N}$, that $n \in R$ is invertible with inverse $n^{-1}$. Define $\mathsf{s}_{\lbrace \ast \rbrace}: \mathbb{N} \times \mathbb{N} \to R$ as follows: 
\[\mathsf{s}_{\lbrace \ast \rbrace}(n,k) = \begin{cases} 0 & \text{ if } n=0 \\
n^{-1} \cdot \delta(n, m+1) & \text{ if } n \geq 1
\end{cases}\]
Before checking the Second Fundamental Theorem of Calculus, let us first examine simply the composite $\mathsf{s}_{\lbrace \ast \rbrace}\mathsf{d}_{\lbrace \ast \rbrace}$: 
\[(\mathsf{s}_{\lbrace \ast \rbrace}\mathsf{d}_{\lbrace \ast \rbrace})(n,k) =\sum \limits_{k \in \mathbb{N}} \mathsf{s}_{\lbrace \ast \rbrace}(n,k) \cdot \mathsf{d}_{\lbrace \ast \rbrace}(k,m) \]
There is only one possible case for when $\mathsf{s}_{\lbrace \ast \rbrace}(n,j) \cdot \mathsf{d}_{\lbrace \ast \rbrace}(j,k) \neq 0$:
\begin{align*}
\mathsf{s}_{\lbrace \ast \rbrace}(n,k) \cdot \mathsf{d}_{\lbrace \ast \rbrace}(k,m) \neq 0 \Leftrightarrow~ n\neq0 \text{ and } n = k+1 \text{ and } k+1=m\Leftrightarrow~ n =m \neq 0 \text{ and } k=n-1
\end{align*}
Hence, $(\mathsf{s}_{\lbrace \ast \rbrace}\mathsf{d}_{\lbrace \ast \rbrace})(n,m) \neq 0$ if and only if $n=m \neq 0$, and in that case we obtain that: 
\begin{align*}
(\mathsf{s}_{\lbrace \ast \rbrace}\mathsf{d}_{\lbrace \ast \rbrace})(n,n) &=~\sum \limits_{k \in \mathbb{N}} \mathsf{s}_{\lbrace \ast \rbrace}(n,k) \cdot \mathsf{d}_{\lbrace \ast \rbrace}(k,n) = \mathsf{s}_{\lbrace \ast \rbrace}(n,n-1) \cdot \mathsf{d}_{\lbrace \ast \rbrace}(n-1,n) = n^{-1} \cdot n = 1
\end{align*}
And so we have that: 
\[(\mathsf{s}_{\lbrace \ast \rbrace}\mathsf{d}_{\lbrace \ast \rbrace})(n,m) = \begin{cases} 0 & \text{ if } n=0 \\
\delta(n,m) & \text{ if } n \geq 1 \end{cases}\]
Now $\oc(0): \mathbb{N} \times \mathbb{N} \to R$ simply checks whether both inputs are zero:
\[\oc(0)(n,m) = \delta(n,0) \cdot \delta(n,m)\]
Therefore if $n=0$, we have that: 
\[(\mathsf{s}_{\lbrace \ast \rbrace}\mathsf{d}_{\lbrace \ast \rbrace})(n,m) + \oc(0)(n,m) = 0 + \delta(n,0) \cdot \delta(n,m) = \delta(n,m) \]
While if $n \neq 0$, we have that:
\[(\mathsf{s}_{\lbrace \ast \rbrace}\mathsf{d}_{\lbrace \ast \rbrace})(n,m) + \oc(0)(n,m) = \delta(n,m) + \delta(n,0) \cdot \delta(n,m) = \delta(n,m) \]
And so we conclude that $\mathsf{s}_{\lbrace \ast \rbrace}\mathsf{d}_{\lbrace \ast \rbrace} + \oc(0) = \delta$, where recall that $\delta$ is the identity in $R^{\Pi}$. 

\begin{theorem} Let $R$ be a commutative complete semiring such that all positive sums of the multiplicative unit are invertible. Then $R^{\Pi}$ is a differential linear category with antiderivatives. 
\end{theorem}

The resulting antiderivative integral transformation $\mathsf{s}_X: \oc X \times (\oc X \times X) \to R$ amounts to pulling out a single element from a bag: 
\[\mathsf{s}_X( \llbracket y_1, \hdots, y_n \rrbracket, (\llbracket x_1, \hdots, x_m \rrbracket, x) ) = \begin{cases} 0 & \text{ if } \llbracket y_1, \hdots, y_n \rrbracket = \llbracket \rrbracket \\ 
n^{-1} \cdot \delta(\llbracket y_1, \hdots, y_n \rrbracket, \llbracket x_1, \hdots, x_m, x \rrbracket) & \text{ o.w. }\end{cases}\]
If $R$ is additively idempotent then $n^{-1} = 1$ and so the antiderivative integral transformation is precisely the coderiving transformation (Definition \ref{dcircdef}). This is the case for $\mathsf{REL}$ \cite[Example 2]{cockett_lemay_2018}. 

\section{Convenient Vector Spaces} \label{consec}

In this section, we show that the differential category of convenient vector spaces \cite{blute2010convenient} has antiderivatives, which is the main goal of this paper. For a detailed introduction to the theory of convenient vector spaces, we invite the reader to see \cite{frolicher1988linear,kriegl1997convenient}. Throughout this section, we follow mostly the terminology and notation used in \cite{blute2010convenient}. 

Recall that a locally convex space is a topological $\mathbb{R}$-vector space (where $\mathbb{R}$ is the reals) which is Hausdorff and such that $0$ has a neighbourhood basis of convex sets or equivalently, an $\mathbb{R}$-vector space with a family of seminorms which separates points (see \cite{treves2016topological} for more details). It should be noted that in some definitions of locally convex spaces, the requirement that the topology is Hausdorff is not necessary. However, following the conventions used in \cite{blute2010convenient,frolicher1988linear,kriegl1997convenient}, we assume that our locally convex spaces are Hausdorff to insure that all derivatives be unique. Playing a fundamental role in the theory of convenient vector spaces is the notion of smooth curves. 

\begin{definition} Let $E$ be a locally convex space. 
\begin{enumerate}[{\em (i)}]
\item A \textbf{curve} is a function $\mathsf{c}: \mathbb{R} \to E$. 
\item A curve $\mathsf{c}: \mathbb{R} \to E$ is \textbf{differentiable} if the limit:
\begin{equation}\label{}\begin{gathered} \lim \limits_{t \to 0} \frac{\mathsf{c}(x+t) - \mathsf{c}(x)}{t}\end{gathered}\end{equation}
exists for all $x \in E$. We define the \textbf{derivative} of $\mathsf{c}$ to be the curve $\mathsf{c}^\prime: \mathbb{R} \to E$ where:
\begin{equation}\label{}\begin{gathered}\mathsf{c}^\prime(x):= \lim \limits_{t \to 0} \frac{\mathsf{c}(x+t) - \mathsf{c}(x)}{t}\end{gathered}\end{equation}
\item A curve is said to be \textbf{smooth} if all its iterated derivatives exists, that is, the curve is infinitely differentiable. Let $\mathcal{C}^\infty(E)$ denote the set of smooth curves of $E$. 
\end{enumerate}
\end{definition}

There are numerous equivalent ways of defining a convenient vector space, see for example \cite[Theorem 2.14]{kriegl1997convenient}. For the purpose of this paper, the main definition of interest is the one which states that every smooth curves admits an antiderivative: 

\begin{definition} A \textbf{convenient vector space} \cite{frolicher1988linear,kriegl1997convenient} is a locally convex space $E$ such that for every smooth curve $\mathsf{c} \in \mathcal{C}^\infty(E)$ there exists a smooth curve $\tilde{\mathsf{c}} \in \mathcal{C}^\infty(E)$ such that $\tilde{\mathsf{c}}^\prime = \mathsf{c}$. We say that $\tilde{\mathsf{c}}$ is an antiderivative of $\mathsf{c}$. 
\end{definition}

One antiderivative in particular is the one provided by Riemann integrals: 

\begin{lemma}\cite[Lemma 2.5]{kriegl1997convenient} \label{intdef} Let $E$ be a convenient vector space. Then for every smooth curve $\mathsf{c} \in \mathcal{C}^\infty(E)$, there exists a unique smooth curve $\int \mathsf{c} \in \mathcal{C}^\infty(E)$ such that $\left( \int \mathsf{c} \right)^\prime = \mathsf{c}$ and $(\int \mathsf{c})(0) = 0$. 
\end{lemma}
\begin{proof} Given any antiderivative $\tilde{\mathsf{c}}$ of $\mathsf{c}$, define $\int \mathsf{c}: \mathbb{R} \to E$ as follows: 
 \begin{equation}\label{}\begin{gathered} \int \mathsf{c} : = \tilde{\mathsf{c}} - \tilde{\mathsf{c}}(0) \end{gathered}\end{equation}
 where $\tilde{\mathsf{c}}(0): \mathbb{R} \to E$ is viewed as a constant smooth curve. This definition is independent of the choice of antiderivative $\tilde{\mathsf{c}}$ \cite{kock1985calculus}, and clearly $\left( \int \mathsf{c} \right)^\prime = \mathsf{c}$ and $(\int \mathsf{c})(0) = 0$. For a more explicit description, $\int \mathsf{c}$ can also be defined as follows: 
\begin{equation}\label{}\begin{gathered} \left( \int \mathsf{c} \right)(r) = \int \limits^r_0 \mathsf{c}(t) ~\mathsf{d}t
 \end{gathered}\end{equation}
 where $ \int \limits^b_a \mathsf{c}(t) ~\mathsf{d}t$ is the standard Riemann integral for topological vector spaces. \end{proof} 

We now wish to define the category of convenient vector spaces. The only remaining question is which maps to take for this category. For convenient vector spaces, there are two important sets of maps: the smooth maps and the bounded linear maps. An equivalent definition of a convenient vector space can be expressed using the bornology of a locally convex space \cite{blute2010convenient,frolicher1988linear,kriegl1997convenient}. Recall that in a locally convex space $E$, a subset $B \subseteq E$ is bounded if for every open subset $U \subseteq E$ containing $0$, there exists a positive real $r > 0$ such that $B \subseteq r \cdot U$. 

\begin{definition} Let $E$ and $F$ be convenient vector spaces. 
\begin{enumerate}[{\em (i)}]
\item A \textbf{bounded linear map}\footnote{It is worth mentioning that a bounded linear map is the same thing as a continuous linear map. } is a linear map $f: E \to F$ which maps bounded sets to bounded sets, that is, if $B \subseteq E$ is bounded, then $f(B) \subseteq F$ is bounded.
\item A \textbf{smooth map} is a function $f: E \to F$ which preserves smooth curves, that is, if $\mathsf{c} \in \mathcal{C}^\infty(E)$ then $\mathsf{c} f \in \mathcal{C}^\infty(F)$. Let $\mathcal{C}^\infty(E, F)$ denote the set of smooth maps between $E$ and $F$. 
\end{enumerate}
\end{definition}

We will soon see that smooth maps in this context are precisely the coKleisli maps of a certain monoidal coalgebra modality. Note that every bounded linear map is smooth \cite{frolicher1988linear,kriegl1997convenient} and since $\mathbb{R}$ is a convenient vector space, that $\mathcal{C}^\infty(E, \mathbb{R}) = \mathcal{C}^\infty(E)$. Furthermore, for every pair of convenient vector spaces $E$ and $F$, $\mathcal{C}^\infty(E, F)$ is also a convenient vector space \cite[Corollary 5.9]{blute2010convenient}. 

Let $\mathsf{CON}$ be the category of convenient vector spaces and bounded linear maps between them. As shown in \cite[Section 4]{blute2010convenient}, $\mathsf{CON}$ is an additive symmetric monoidal closed category where the additive structure is given by biproducts, the tensor product is given by the Mackey completion \cite[Lemma 2.2]{kriegl1997convenient} of the algebraic tensor product, and the monoidal unit is $\mathbb{R}$. We should note that while Mackey completion plays an important role in the theory of convenient vector spaces, it is not crucial to the understanding of how to obtain an integral and antiderivatives in the differential category context. For the purpose of this paper, one only needs to understand Lemma \ref{intdef} and that there is a bijective correspondence between smooth maps and bounded linear maps which involves the smallest convenient vector space containing the image of a certain evaluation map. For more details on Mackey completeness and the category $\mathsf{CON}$, see \cite{blute2010convenient,frolicher1988linear,kriegl1997convenient}. 

We will now give an overview of the differential linear category structure of $\mathsf{CON}$. For every convenient vector space $E$, let $E^\ast := \mathsf{CON}(E, \mathbb{R})$ denote the set of bounded linear functionals. Define the smooth map $\mathsf{ev}_E: E \to \mathcal{C}^\infty(E)^\ast$ \cite[Lemma 6.1]{blute2010convenient} as the evaluation map: 
\begin{equation}\label{}\begin{gathered} \mathsf{ev}_E(x)(\mathsf{c}) := \mathsf{c}(x)
\end{gathered}\end{equation}
Then define $\oc E$ \cite[Definition 6.2]{blute2010convenient} as the Mackey completion of the image $\mathsf{ev}_E(x)$ in $\mathcal{C}^\infty(E)^\ast$, in other words, ${\oc E \subset \mathcal{C}^\infty(E)^\ast}$ is the smallest convenient vector space which contains $\mathsf{ev}_E(E)$. Define the resulting induced smooth map $\delta_E: E \to \oc E$ as:
\begin{equation}\label{}\begin{gathered} \delta_E(x) := \mathsf{ev}_E(x)
\end{gathered}\end{equation}
This gives a coalgebra modality $\oc$ on $\mathsf{CON}$ which satisfies the Seely isomorphisms \cite[Lemma 6.4]{blute2010convenient}:
\[\oc(E \times F) \cong \oc E \otimes \oc F \quad \quad \oc(0) \cong \mathbb{R}\]
and therefore is also a monoidal coalgebra modality (for full details see \cite[Section 6]{blute2010convenient}). Furthermore, as promised, the coKleisi maps of this coalgebra modality are precisely the smooth maps between convenient vector spaces \cite[Theorem 6.3]{blute2010convenient}:
\begin{equation}\label{}\begin{gathered} \mathsf{CON}(\oc E, F) \cong \mathcal{C}^\infty(E, F) \end{gathered}\end{equation}
In particular, the isomorphism in the direction $\mathsf{CON}(\oc E, F) \to \mathcal{C}^\infty(E, F)$ is given by precomposing with $\delta_E$. This implies that for every smooth map $f: E \to F$ there exists a unique bounded linear map $g: \oc E \to F$ such that the following diagram commutes: 
\begin{equation}\label{}\begin{gathered} \xymatrixcolsep{5pc}\xymatrix{E \ar[r]^-{\delta_E} \ar[dr]_-{f} & \oc E \ar[d]^-{g} \\  
 & F
 } \end{gathered}\end{equation}
The deriving transformation $\mathsf{d}_E: \oc E \otimes E \to \oc E$ is given by differentiating smooth maps in the classical sense \cite[Proposition 5.12]{blute2010convenient}. In particular, one has that: 
\begin{equation}\label{}\begin{gathered} \mathsf{d}_E\left(\delta_E(x) \otimes y \right) := \lim \limits_{t \to 0} \frac{\delta_E(x + t \cdot y) - \delta(x)}{t} \end{gathered}\end{equation}
To help us understand the derivative of a smooth maps, note that every bounded linear map $g: \oc E \otimes E \to F$ can equivalently be described as a smooth map $\overline{g}: E \times E \to F$ which is linear in its second argument. Explicitly, $\overline{g}$ is the unique smooth map such that the following diagram commutes:
\begin{equation}\label{ddef}\begin{gathered} \xymatrixcolsep{5pc}\xymatrix{E \times E \ar[r]^-{\delta_E} \ar[ddrrr]_-{\overline{g}} & \oc (E \times E) \ar[r]^-{\chi} & \oc E \otimes \oc E \ar[r]^-{1_{\oc E} \otimes \varepsilon_E} & \oc E \otimes E \ar[dd]^-{g} \\ \\
 & && F
 } \end{gathered}\end{equation}
 where recall that $\chi$ is the Seely isomorphism (Definition \ref{Seelydef}). Then for a smooth map $f: E \to F$, its derivative $\mathsf{D}[f] = \mathsf{d}_E f: \oc E \otimes E \to \oc E$ can be seen as smooth map $\overline{\mathsf{D}[f]}: E \times E \to F$ which is linear in its second argument and given by: 
\begin{equation}\label{}\begin{gathered} \overline{\mathsf{D}[f]}(x, y) := \lim \limits_{t \to 0} \frac{f(x + t \cdot y) - f(x)}{t} \end{gathered}\end{equation}
Note that this is the standard definition of the derivative in multivariable differential calculus. This is also precisely the Cartesian differential category structure of the coKleisli category of $\oc$ \cite{blute2009cartesian}. In this case, the coKleisli category of $\oc$ is isomorphic to the category of convenient vector spaces and smooth maps between them. 

We will now show that we have antiderivatives in the differential category context, that is, we wish to apply Theorem \ref{antithm}. So we turn our attention to the monoidal unit. For the monoidal unit $\mathbb{R}$, $\delta_\mathbb{R}: \mathbb{R} \to \oc \mathbb{R}$ is a smooth curve -- which makes our work much easier since smooth curves behave very nicely for convenient vector spaces. One can check that its derivative $\delta^\prime_\mathbb{R}: \mathbb{R} \to \oc \mathbb{R}$ is given by evaluating derivatives of smooth curves:
\[\delta^\prime_\mathbb{R}(r)(\mathsf{c}) = \mathsf{c}^\prime(r) \quad \quad \mathsf{c} \in \mathcal{C}^\infty(\mathbb{R})\]
As a result, the deriving transformation $\mathsf{d}_\mathbb{R}: \oc \mathbb{R} \to \oc \mathbb{R}$ is the unique bounded linear map such that the following diagram commutes: 
\begin{equation}\label{sdef}\begin{gathered} \xymatrixcolsep{5pc}\xymatrix{\mathbb{R} \ar[r]^-{\delta_\mathbb{R}} \ar[dr]_-{\delta_\mathbb{R}^\prime} & \oc \mathbb{R} \ar[d]^-{\mathsf{d}_\mathbb{R}} \\  
 &\oc \mathbb{R}
 } \end{gathered}\end{equation}
To obtain the desired integral $\mathsf{s}_\mathbb{R}$, we apply Lemma \ref{intdef} to $\delta_\mathbb{R}$ to obtain its special antiderivative $\int \delta_\mathbb{R}: \mathbb{R} \to \oc \mathbb{R}$. By uniqueness of this antiderivative, one can easily check that $\int \delta_\mathbb{R}$ is given by evaluating antiderivative of smooth curves:
\[\left(\int \delta_\mathbb{R} \right)(r)(\mathsf{c}) =\left(\int \mathsf{c} \right)(r) \quad \quad \mathsf{c} \in \mathcal{C}^\infty(\mathbb{R})\]
Define $\mathsf{s}_\mathbb{R}: \oc \mathbb{R} \to \oc \mathbb{R}$ as the unique bounded linear map such that the following diagram commutes: 
\begin{equation}\label{}\begin{gathered} \xymatrixcolsep{5pc}\xymatrix{\mathbb{R} \ar[r]^-{\delta_\mathbb{R}} \ar[dr]_-{\int\delta_\mathbb{R}} & \oc \mathbb{R} \ar[d]^-{\mathsf{s}_\mathbb{R}} \\  
 &\oc \mathbb{R}
 } \end{gathered}\end{equation}
As usual, $\oc(0): \oc \mathbb{R} \to \oc \mathbb{R}$ is given by evaluating at zero, that is, $\oc(0)$ is the unique bounded linear map such that the following diagram commutes: 
\begin{equation}\label{!0def}\begin{gathered} \xymatrixcolsep{5pc}\xymatrix{\mathbb{R} \ar[r]^-{\delta_\mathbb{R}} \ar[dr]_-{\delta_\mathbb{R}(0)} & \oc \mathbb{R} \ar[d]^-{\oc(0)} \\  
 &\oc \mathbb{R}
 } \end{gathered}\end{equation}
That $\mathsf{d}_\mathbb{R}$ and $\mathsf{s}_\mathbb{R}$ satisfy the Second Fundamental Theorem of Calculus, that is, $\mathsf{s}_\mathbb{R}$ satisfies \textbf{[ftc.2]}, follows mostly from the fact that the Second Fundamental Theorem of Calculus holds in the convenient vector space context. 

\begin{lemma}\cite[Corollary 2.6.(6)]{kriegl1997convenient} \label{intFT2} Let $E$ be a convenient vector space. Then for every smooth curve $\mathsf{c} \in \mathcal{C}^\infty(E)$, the following equality holds: 
\begin{equation}\label{}\begin{gathered} \int \mathsf{c}^\prime = \mathsf{c} - \mathsf{c}(0) \end{gathered}\end{equation}
where $\mathsf{c}(0): \mathbb{R} \to E$ is viewed as a constant smooth function. 
\end{lemma}

We will also need the following lemma which allows us to pull bounded linear maps in and out of antiderivatives. 
  
\begin{lemma} \label{kock1} \cite[Proposition 2.3]{kock1985calculus} Let $E$ be a convenient vector space. Then for every smooth curve $\mathsf{c} \in \mathcal{C}^\infty(E)$ and bounded linear map $f: E \to F$, the following equality holds: 
\begin{equation}\label{}\begin{gathered} \int (\mathsf{c}f) = \left( \int \mathsf{c} \right) f
 \end{gathered}\end{equation}
\end{lemma}  

Finally, to show that $\mathsf{s}_\mathbb{R} \mathsf{d}_\mathbb{R} + \oc(0) = 1_{\oc \mathbb{R}}$, it suffices to show that $\delta_\mathbb{R}\mathsf{s}_\mathbb{R}\mathsf{d}_\mathbb{R} = \delta_\mathbb{R} - \delta_\mathbb{R}\oc(0)$. 
\begin{align*}
\delta_\mathbb{R}\mathsf{s}_\mathbb{R}\mathsf{d}_\mathbb{R} &=~\left( \int \delta \right) \mathsf{d}_\mathbb{R}\tag{\ref{sdef}} \\
&=~\left( \int \delta \mathsf{d}_\mathbb{R} \right) \tag{Lemma \ref{kock1}} \\
&=~ \int \delta^\prime_\mathbb{R} \tag{\ref{ddef}} \\
&=~\delta_\mathbb{R} - \delta_\mathbb{R}(0) \tag{Lemma \ref{intFT2}}\\
&=~ \delta_\mathbb{R} - \delta_\mathbb{R} \oc(0)\tag{\ref{!0def}}
\end{align*}
And so we conclude that: 
\begin{theorem}\label{mainthm} $\mathsf{CON}$ is a differential linear category with antiderivatives. 
\end{theorem} 

The antiderivative integral transformation $\mathsf{s}_E: \oc E \to \oc E \otimes E$ is the unique bounded linear map which when precomposing by $\delta_E: E \to \oc E$ gives the following equality:  
\[\mathsf{s}_E(\delta_E(x)) = \left( \int \limits^1_0 \delta_E(t \cdot x) ~ \mathsf{d}t \right) \otimes x \]
where $\int \limits^1_0 \delta_E(t \cdot x) ~ \mathsf{d}t $ is Riemann integral of the smooth curve $\delta_E(- \cdot x) := \mathbb{R} \to E$ (similar to the integral consider in \cite[Proposition 3.1]{kock1985calculus}). Recall that every bounded linear map $f: \oc E \otimes E \to F$ can be seen as a smooth map $\overline{f}: E \times E \to F$ which is linear in its second argument. Therefore, its integral $\mathsf{S}[f]: \oc E \to F$ is the unique bounded linear map which when precomposing by $\delta_E$ gives: 
\[\mathsf{S}[f](\delta_E(x)) = \int\limits^1_0 \overline{f}(tx, x) ~\mathsf{d}t\]
Note that this integral is the same as the one discussed for smooth functions in \cite{COCKETT201845,cruttwell2019integral}. In particular, this integral satisfies the Rota-Baxter rule and the Second Fundamental Theorem of Calculus. 

\bibliographystyle{spmpsci}      
\bibliography{convenientbib}   

\begin{thebibliography}{10}
\providecommand{\url}[1]{{#1}}
\providecommand{\urlprefix}{URL }
\expandafter\ifx\csname urlstyle\endcsname\relax
  \providecommand{\doi}[1]{DOI~\discretionary{}{}{}#1}\else
  \providecommand{\doi}{DOI~\discretionary{}{}{}\begingroup
  \urlstyle{rm}\Url}\fi

\bibitem{bierman1995categorical}
Bierman, G.M.: What is a categorical model of intuitionistic linear logic?
\newblock In: International Conference on Typed Lambda Calculi and
  Applications, pp. 78--93. Springer (1995)

\bibitem{blute2015cartesian}
Blute, R., Cockett, J.R.B., Seely, R.A.G.: Cartesian differential storage
  categories.
\newblock Theory and Applications of Categories \textbf{30}(18), 620--686
  (2015)

\bibitem{Blute2019}
Blute, R.F., Cockett, J.R.B., Lemay, J.S.P., Seely, R.A.G.: Differential
  categories revisited.
\newblock Applied Categorical Structures  (2019).
\newblock \doi{10.1007/s10485-019-09572-y}

\bibitem{blute2006differential}
Blute, R.F., Cockett, J.R.B., Seely, R.A.G.: Differential categories.
\newblock Mathematical structures in computer science \textbf{16}(06),
  1049--1083 (2006)

\bibitem{blute2009cartesian}
Blute, R.F., Cockett, J.R.B., Seely, R.A.G.: Cartesian differential categories.
\newblock Theory and Applications of Categories \textbf{22}(23), 622--672
  (2009)

\bibitem{blute2010convenient}
Blute, R.F., Ehrhard, T., Tasson, C.: A convenient differential category.
\newblock Cahiers de Top. et G{\'e}om Diff \textbf{LIII}, 211--232 (2012)

\bibitem{COCKETT201845}
Cockett, J., Lemay, J.S.P.: Cartesian integral categories and contextual
  integral categories.
\newblock Electronic Notes in Theoretical Computer Science \textbf{341}, 45 --
  72 (2018).
\newblock \doi{https://doi.org/10.1016/j.entcs.2018.11.004}.
\newblock Proceedings of the Thirty-Fourth Conference on the Mathematical
  Foundations of Programming Semantics (MFPS XXXIV)

\bibitem{cockett_lemay_2018}
Cockett, J.R.B., Lemay, J.S.P.: Integral categories and calculus categories.
\newblock Mathematical Structures in Computer Science pp. 1--66 (2018).
\newblock \doi{10.1017/S0960129518000014}

\bibitem{cruttwell2019integral}
Cruttwell, G., Lemay, J.S.P., Lucyshyn-Wright, R.: Integral and differential
  structure on the free $c^\infty$-ring modality.
\newblock arXiv preprint arXiv:1902.04555  (2019)

\bibitem{ehrhard2017introduction}
Ehrhard, T.: An introduction to differential linear logic: proof-nets, models
  and antiderivatives.
\newblock Mathematical Structures in Computer Science pp. 1--66 (2017)

\bibitem{ehrhard2003differential}
Ehrhard, T., Regnier, L.: The differential lambda-calculus.
\newblock Theoretical Computer Science \textbf{309}(1), 1--41 (2003)

\bibitem{fiore2007differential}
Fiore, M.: Differential structure in models of multiplicative biadditive
  intuitionistic linear logic.
\newblock In: International Conference on Typed Lambda Calculi and
  Applications, pp. 163--177. Springer (2007)

\bibitem{frolicher1988linear}
Fr\"{o}licher, A., Kriegl, A.: Linear spaces and differentiation theory.
\newblock Pure and Applied Mathematics  (1988)

\bibitem{givant2008introduction}
Givant, S., Halmos, P.: Introduction to Boolean algebras.
\newblock Springer Science \& Business Media (2008)

\bibitem{golan2013semirings}
Golan, J.S.: Semirings and their Applications.
\newblock Springer Science \& Business Media (2013)

\bibitem{guo2012introduction}
Guo, L.: An introduction to Rota-Baxter algebra, vol.~2.
\newblock International Press Somerville (2012)

\bibitem{hyland2003glueing}
Hyland, M., Schalk, A.: Glueing and orthogonality for models of linear logic.
\newblock Theoretical computer science \textbf{294}(1-2), 183--231 (2003)

\bibitem{kock1985calculus}
Kock, A.: Calculus of smooth functions between convenient vector spaces.
\newblock Citeseer (1985)

\bibitem{kriegl1997convenient}
Kriegl, A., Michor, P.W.: The convenient setting of global analysis, vol.~53.
\newblock American Mathematical Soc. (1997)

\bibitem{laird2013weighted}
Laird, J., Manzonetto, G., McCusker, G., Pagani, M.: Weighted relational models
  of typed lambda-calculi.
\newblock In: Proceedings of the 2013 28th Annual ACM/IEEE Symposium on Logic
  in Computer Science, pp. 301--310. IEEE Computer Society (2013)

\bibitem{lamarche1992quantitative}
Lamarche, F.: Quantitative domains and infinitary algebras.
\newblock Theoretical computer science \textbf{94}(1), 37--62 (1992)

\bibitem{lang2002algebra}
Lang, S.: Algebra, revised 3rd ed.
\newblock Graduate Texts in Mathematics \textbf{211} (2002)

\bibitem{mac2013categories}
Mac~Lane, S.: Categories for the working mathematician.
\newblock Springer-Verlag, New York, Berlin, Heidelberg (1971, revised 2013)

\bibitem{mellies2003categorical}
Mellies, P.A.: Categorical models of linear logic revisited  (2003)

\bibitem{mellies2009categorical}
Mellies, P.A.: Categorical semantics of linear logic.
\newblock Panoramas et syntheses \textbf{27}, 15--215 (2009)

\bibitem{mellies2017explicit}
Melli{\`e}s, P.A., Tabareau, N., Tasson, C.: An explicit formula for the free
  exponential modality of linear logic.
\newblock Mathematical Structures in Computer Science pp. 1--34 (2017)

\bibitem{ong2017quantitative}
Ong, C.H.L.: Quantitative semantics of the lambda calculus: Some
  generalisations of the relational model.
\newblock In: Logic in Computer Science (LICS), 2017 32nd Annual ACM/IEEE
  Symposium on, pp. 1--12. IEEE (2017)

\bibitem{schalk2004categorical}
Schalk, A.: What is a categorical model of linear logic.
\newblock Manuscript, available from http://www. cs. man. ac. uk/~ schalk/work.
  html  (2004)

\bibitem{treves2016topological}
Treves, F.: Topological Vector Spaces, Distributions and Kernels: Pure and
  Applied Mathematics, vol.~25.
\newblock Elsevier (2016)

\end{thebibliography}

\end{document}